\documentclass[a4paper,10pt]{amsart}
\usepackage{amssymb,epsfig}
\usepackage{latexsym}
\usepackage{amsthm}
\usepackage{amsmath}
\usepackage{amscd}
\input xy
\xyoption{all}
\newtheorem{oss}{Remark}
\newtheorem{esempio}{Example}
\newtheorem{deff}{Definition}
\newtheorem{prop}{Proposition}
\newtheorem{teo}{Theorem}
\newtheorem{cor}{Corollary}
\newtheorem{lem}{Lemma}
\newcommand{\cc}{\mathbb{C}}
\newcommand{\pp}{\mathbb{P}}
\newcommand{\nn}{\mathbb{N}}
\newcommand{\qq}{\mathbb{Q}}

\newcommand{\efascio}{\mathcal{E}}
\newcommand{\efil}{E^{\bullet}}
\newcommand{\alphafil}{\underline{\alpha}}
\newcommand{\filtrazione}{\efil,\alphafil}
\newcommand{\muno}{\mu_{\rho_{a_1,b_1}}(\filtrazione;\varphi_1)}
\newcommand{\mudu}{\mu_{\rho_{a_2,b_2}}(\filtrazione;\varphi_2)}
\newcommand{\munoc}{\mu_{\rho_{a_1,b_1,c_1}}(\filtrazione;\varphi_1)}
\newcommand{\muduc}{\mu_{\rho_{a_2,b_2,c_2}}(\filtrazione;\varphi_2)}
\newcommand{\mui}{\mu_{\rho_{a_i,b_i}}(\filtrazione;\varphi_i)}
\newcommand{\muic}{\mu_{\rho_{a_i,b_i,c_i}}(\filtrazione;\varphi_i)}
\newcommand{\mudus}{\mu_{\rho'}(\filtrazione;[\phitilde,\omega])}
\newcommand{\isom}{\textit{Isom}}
\newcommand{\home}{\textit{Hom}}
\newcommand{\symstar}{\text{Sym}^{\star}}
\newcommand{\rk}[1]{\text{rk}({#1})}
\newcommand{\zaza}[1]{\text{End}({#1})}
\newcommand{\omegax}{\Omega^1_X}
\newcommand{\omegaxt}{\Omega^1_{\widetilde{X}}}
\newcommand{\ox}{\mathcal{O}_X}
\newcommand{\oxtilde}{\mathcal{O}_{\widetilde{X}}}
\newcommand{\xtilde}{\widetilde{X}}
\newcommand{\tautilde}{\widetilde{\tau}}
\newcommand{\phitilde}{\widetilde{\phi}}

\newcommand{\rest}[2]{{#1}_{\mid_{#2}}}
\newcommand{\degparalpha}{\alpha\text{-deg}_{\text{\tiny{par}}}}
\newcommand{\degpar}{\text{deg}_{\text{\tiny{par}}}}
\newcommand{\mupar}{\mu_{\text{\tiny{par}}}}
\newcommand{\ab}[1]{_{a_{#1},b_{#1}}}
\newcommand{\abc}[1]{_{a_{#1},b_{#1},c_{#1}}}
\newcommand{\ddgpb}{$2$\text{-dgpb }}
\newcommand{\dgpb}{\text{dgpb }}
\newcommand{\moduli}{\mathfrak{\underline{M}}(\rho)^{(\delta_1,\delta_2)\mbox{-}(s)s}}
\newcommand{\modulis}{\mathfrak{\underline{M}}(\rho)^{(\delta_1,\delta_2)\mbox{-}s}}
\newcommand{\moduliss}{\mathfrak{\underline{M}}(\rho)^{(\delta_1,\delta_2)\mbox{-}ss}}
\newcommand{\modulispaces}{\mathcal{M}(\rho)^{(\delta_1,\delta_2)\mbox{-}s}}
\newcommand{\modulispacess}{\mathcal{M}(\rho)^{(\delta_1,\delta_2)\mbox{-}ss}}

\newcommand{\unomodulispaces}{\mathcal{M}(\rho)^{(\delta)\mbox{-}s}}
\newcommand{\unomodulispacess}{\mathcal{M}(\rho)^{(\delta)\mbox{-}ss}}
\newcommand{\singprincipalmoduliss}{\underline{\mathfrak{M}}(\rho)^{ss}_{r}}
\newcommand{\singprincipalmodulis}{\underline{\mathfrak{M}}(\rho)^{s}_{r}}
\newcommand{\singprincipalmodulispacess}{\mathcal{M}(\rho)^{ss}_{r}}
\newcommand{\singprincipalmodulispaces}{\mathcal{M}(\rho)^{s}_{r}}
\newcommand{\unomodulispacessfr}{\mathcal{M}(\rho)^{(\textbf{\tiny{fr-}}\delta)\mbox{-}ss}}
\newcommand{\unomodulispacesfr}{\mathcal{M}(\rho)^{(\textbf{\tiny{fr-}}\delta)\mbox{-}s}}
\newcommand{\unomodulisfr}{\mathfrak{\underline{M}}(\rho)^{(\textbf{\tiny{fr-}}\delta)\mbox{-}s}}
\newcommand{\unomodulissfr}{\mathfrak{\underline{M}}(\rho)^{(\textbf{\tiny{fr-}}\delta)\mbox{-}ss}}
\newcommand{\muphipar}[2]{\mu_{\text{\tiny{par}}}({#1},{#2})}
\newcommand{\gr}{\text{gr}}

\newcommand{\ssfam}{\mathfrak{F}^{\delta\text{-ss}}_{d,r,a,b,L}}
\newcommand{\quoz}{\mathfrak{h}}
\newcommand{\quot}{\mathfrak{H}}
\newcommand{\gras}{\textbf{Gr}}
\newcommand{\poly}[1]{P_{({#1},\rest{q}{{#1}},\rest{\varphi}{{#1}})}}
\newcommand{\polygen}[3]{P_{({#1},{#2},{#3})}}
\newcommand{\polyf}{P_{(E,q,\varphi)}}
\newcommand{\qdim}[1]{\chi({#1})}
\newcommand{\tipo}{\underline{\mathfrak{t}}}
\newcommand{\kk}[2]{\texttt{k}_{\tiny{{#1}{#2}}}}
\newcommand{\fr}{\text{\tiny{fr}}}
\title[Principal Higgs bundles over singular curve]{A compactification of the moduli space of principal Higgs bundles over singular curves}
\author{Alessio Lo Giudice}
\address{\scriptsize Alessio Lo giudice:
Scuola Internazionale Superiore di Studi Avanzati,
Via Bonomea 265, 34136 Trieste, Italia}
\email{alessio.logiudice@sissa.it}
\author{Andrea Pustetto}
\address{\scriptsize Andrea Pustetto:
Scuola Internazionale Superiore di Studi Avanzati,
Via Bonomea 265, 34136 Trieste, Italia}
\email{andrea.pustetto@sissa.it}
\begin{document}
\maketitle
\begin{abstract}
A principal Higgs bundle $(P,\phi)$ over a singular curve $X$ is a pair consisting of a principal bundle $P$ and a morphism $\phi:X\to\text{Ad}P \otimes \Omega^1_X$. We construct the moduli space of principal Higgs G-bundles  over an irreducible singular curve $X$ using the theory of decorated vector bundles. More precisely, given a faithful representation $\rho:G\to Sl(V)$ of $G$, we consider principal Higgs bundles as triples $(E,q,\varphi)$ where $E$ is a vector bundle with $\rk{E}=\dim V$ over the normalization $\xtilde$ of $X$, $q$ is a parabolic structure on $E$ and $\varphi:E\ab{}\to L$ is a morphism of bundles, being $L$ a line bundle and $E\ab{}\doteqdot (E^{\otimes a})^{\oplus b}$ a vector bundle depending on the Higgs field $\phi$ and on the principal bundle structure. Moreover we show that this moduli space for suitable integers $a,b$ is related to the space of framed modules.
\end{abstract}

\section*{Introduction}
Let X be an irreducible curve over $\cc$ and $G$ a reductive complex algebraic group. Recall that a Higgs vector bundle $\mathfrak E$ is a pair $(E,\phi)$ where $E$ is a vector bundle over $X$ and $\phi:E\to E\otimes \omegax$ is a morphism of vector bundles; if $\dim X>1$, one requires that $\phi\wedge\phi =0$. Higgs bundles are important in many areas of mathematics and mathematical physics. For example, it was shown by Hitchin that their moduli spaces give examples of Hyper-K\"ahler manifolds (see \cite{hitchin1987self}).
It is possible to generalize Higgs vector bundles to principal objects. A principal Higgs $G$-bundle $\mathfrak P$ is a pair $(P,\phi)$ where $P$ is a principal $G$-bundle over $X$ and $\phi:X\to\text{Ad}P\otimes \omegax$ a section such that $[\phi,\phi]=0$. If the curve $X$ is smooth then the moduli space of semistable principal Higgs has been widely studied in the last years (see \cite{nitsure1991moduli} or \cite{simpson1994moduli}).\\
In this paper we consider principal Higgs bundles over singular curves and construct their moduli spaces. Since the cuspidal and nodal cases are similar, for simplicity we assume that is $X$ an irreducible nodal curve over $\cc$. If we fix a faithful representation $\rho:G\to Sl(V)$, we can consider a principal $G$-bundle $P$ as a pair $(E,\tau)$ where $E$ is a vector bundle over $X$ and $\tau$ is a morphism of $\ox$ algebras induced by the section $\sigma: X\to E/G$. Then in order to obtain a projective moduli space we enlarge our category allowing $E$ to be a torsion free sheaf $\efascio$, thus obtaining a pair $(\efascio,\tau)$ called a singular principal $G$-bundle. In \cite{bhosle1992generalised} Bhosle showed that the categories of torsion free sheaves on a nodal curve $X$ and of generalized parabolic bundles over its normalization $\xtilde$ are equivalent. In view of this result, we can focus our attention on the second category. Then we define descending principal Higgs bundles as quadruples $(E,q,\tautilde,\phitilde)$ where $E$ is a vector bundle over $\xtilde$, while $q$, $\tautilde$ and $\phitilde$ are morphisms from which we can reconstruct the principal bundle structure and the Higgs field. In order to study these descending bundles we further enlarge our category and so the final objects we shall consider will be triples $(E,q,\varphi)$, called decorated generalized parabolic bundles, where $(E,q)$ is a generalized parabolic vector bundle over $\xtilde$ and $\varphi:(E^{\otimes a})^{\oplus b}\to L$ is a morphism that generalize $\tautilde$ and $\phitilde$. To construct the moduli space we use results of Schmitt. We introduce a notion of semistability for decorated bundles that looks like the one introduced by Huybrecths and Lehn for framed sheaves \cite{huybrechts1995framed} and develop a machinery similar to Huybrecths and Lehn's one. We construct also Jordan-H\"older filtrations for these objects and so we can define $S$-equivalence for decorated generalized parabolic bundles, i.e. we stipulate that $E\simeq E'$ if and only if $gr(E)\simeq gr(E')$ (see Section \ref{sec-j-h-fil}).\\
Our main result is
\begin{teo}\label{teo-main}
There is a projective scheme $\singprincipalmodulispacess$ which co-represents the functor
\begin{align*}
\singprincipalmoduliss :\; & \textnormal{Sch}_{\cc}\longrightarrow\textnormal{Sets}\\
                         & \quad S\longmapsto\left\lbrace\begin{array}{c}
                                                   \textnormal{Isomorphism classes of families of semistable}\\
                                                   \textnormal{singular principal Higgs } G\textnormal{-bundles}\\
                                                   \textnormal{with trivial determinant and rank }r\\
                                                   \textnormal{parametrized by }S
                                                   \end{array}
\right\rbrace
\end{align*}
Moreover there exists an open subscheme $\singprincipalmodulispaces$ which represents the subfunctor $\singprincipalmodulis$ that sends a scheme $S$ to the set of the isomorphism classes of families of stable singular principal Higgs $G$-bundles parametrized by $S$. A closed point in $\singprincipalmodulispacess$ represents an $S$-equivalence class of singular principal Higgs $G$-bundles.
\end{teo}

In Section \ref{sec-descending} we show that we can treat a principal Higgs $G$-bundle over a nodal curve as a particular type of vector bundle on the normalization of the curve called a descending bundle. In Section \ref{secdecoration-semistability} we define decorated vector bundles, introduce a notion of (semi)stability for them and show that all the definitions of (semi)stability that we give agree. In Section \ref{secmodulispace} we construct the moduli space of decorated bundles and give a proof for Theorem \ref{teo-main}; moreover we study semistable $n+1$-uples $(E,\phi_1,\dots,\phi_n)$ where $E$ is a vector bundle over a curve and $\phi_i:E\to L_i$ where $L_i$ are line bundles over $X$, generalizing semistable pairs studied by Nitsure in \cite{nitsure1991moduli}. In the last section we introduce a new notion of semistability for decorated bundles, define Jordan-H\"older filtrations with respect to this semistability and construct the moduli space.

\section{Descending parabolic Higgs $G$-bundles}\label{sec-descending}
Let $X$ be an irreducible nodal curve over $\cc$ and, for sake of convenience, assume that $X$ has only a simple node $x_0\in X$. Let $\nu:\xtilde\to X$ be the normalization map and let $x_1,x_2\in\nu^{-1}(x_0)$. Moreover let $G$ be a reductive algebraic group over $\cc$ and $\rho:G\to Sl(V)$ a faithful representation of $G$.
\begin{oss}
If $G$ is semisimple every faithful representation $\rho:G\to Gl(V)$ is such that $\rho(G)\subseteq Sl(V)$, indeed $\det :\rho(G)\to\cc^*$ is a character and so, since $G$ is semisimple, it is trivial.
\end{oss}
\begin{deff}[\textbf{Principal Higgs Bundle}]
A principal Higgs $G$-bundle $\mathfrak{E}$ over $X$ is a pair $(P,\phi)$ consisting of a principal $G$-bundle $P$ over $X$ and a section $\phi:X\to Ad(P)\otimes\omegax$, where $Ad(P)\doteqdot P\times_{\text{Ad}}\mathfrak{g}$.
\end{deff}
\subsection{Singular principal Higgs $G$-bundles}
Let $P\stackrel{G}{\longrightarrow}X$ be a principal $G$-bundle and $\phi:X\to \text{Ad}(P)\otimes\omegax$ a Higgs field. Thanks to the representation $\rho$ we can associate to $P$ a vector bundle $E\doteqdot P_{\rho}=P\times_{\rho}V$ over $X$; the inclusion of $P$ in the $Gl(V)$-bundle $\isom(V\otimes\ox,E^{\vee})$ associated to $E$ gives a section $\sigma:X\to\isom(V\otimes\ox,E^{\vee})/G$ as follows:
\begin{equation}\label{diag2}
\xymatrix{ P \; \ar[d]^G\ar@{^{(}->}[r] & \isom(V\otimes\ox,E^{\vee}) \; \ar[d]^G\ar@{^{(}->}[r] & \home(V\otimes\ox,E^{\vee})\ar[d]^G\\
X \ar[r]^(.25){\sigma} & \isom(V\otimes\ox,E^{\vee})/G \; \ar@{^{(}->}[r] & \home(V\otimes\ox,E^{\vee})//G
}
\end{equation}
Since $Spec(\symstar(E\otimes V)^G)=\home(V\otimes\ox,E^{\vee})//G$ and $Spec(\ox)=X$, the section $X\to\home(V\otimes\ox,E^{\vee})//G$ induces a morphism of $\ox$-algebras:
\begin{equation*}
\tau:\symstar(E\otimes V)^G\to\ox.
\end{equation*}
The Higgs field $\phi:X\to Ad(P)\otimes\omegax$ induces a section $\phi:X\to\zaza{E}\otimes\omegax$, that we call again $\phi$ for simplicity.\\
Vice versa, given a vector bundle $E$ with trivial determinant and a morphism of $\ox$-algebras $\tau:\symstar(E\otimes V)^G\to\ox$ such that the induced section $\sigma:X\to\home(V\otimes\ox,E^{\vee})//G$ has image in $\isom(V\otimes\ox,E^{\vee})/G$, we have that $\sigma^*\isom(V\otimes\ox,E^{\vee})$ is a principal $G$-bundle over $X$:
\begin{equation*}
\xymatrix{\sigma^*\isom(V\otimes\ox,E^{\vee}) \ar[d]^G \ar[r] & \isom(V\otimes\ox,E^{\vee}) \ar[d]^G\ar@{<~>}[r] & E \ar[d]^G\\
X \ar[r]_(.35){\sigma} & \isom(V\otimes\ox,E^{\vee})/G \ar@{<~>}[r] & E/G.
}
\end{equation*}
Moreover given a section $\phi:X\to\zaza{E}\otimes\omegax$ such that the corresponding morphism $\phi\colon X\to Ad(\isom(V\otimes\ox,E^{\vee}))\otimes\omegax$ has image in $Ad(\sigma^*\isom(V\otimes\ox,E^{\vee}))\otimes\omegax$, then $\phi$ induces a Higgs field also over $\sigma^*\isom(V\otimes\ox,E^{\vee})$.
So, for any fixed faithful representation $\rho:G\to Sl(V)$, there is a one to one correspondence between:
\begin{enumerate}
\item Principal Higgs $G$-bundles over $X$;
\item Triples $(E,\tau,\phi)$ with
\begin{itemize}
\item[-] $E=$ locally free sheaf with $\det E\simeq\ox$;
\item[-] $\tau:\symstar(E\otimes V)^G\to\ox$ morphism of $\ox$-algebras such that the induced section $\sigma:X\to\home(V\otimes\ox,E^{\vee})//G$ has image in $\isom(V\otimes\ox,E^{\vee})/G$;
\item[-] $\phi:X\to \zaza{E}\otimes\omegax$ such that induces a morphism $X\to Ad(\sigma^*\isom(V\otimes\ox,E^{\vee}))\otimes\omegax$.
\end{itemize}
\end{enumerate}
All this leads us to the following definition:

\begin{deff}[\textbf{Singular principal Higgs $G$-bundles}]
We will say that a triple $(\efascio,\tau,\phi)$ is a singular principal Higgs $G$-bundle if:
\begin{itemize}
\item[-] $\efascio$ is a torsion free sheaf;
\item[-] $\tau:\symstar(\efascio\otimes V)^G\to\ox$ is a morphism of $\ox$-algebras;
\item[-] $\phi:X\to\zaza{\efascio}\otimes\omegax$ is a section.
\end{itemize}
Let $\sigma:X\to\home(V\otimes\ox,\efascio^{\vee})//G$ be the section induced by $\tau$ and let $U_{\efascio}$ be the open subset of $X$ in which $\efascio$ is locally free, i.e. $U_{\efascio}=X$ if $\efascio$ is a vector bundle and $U_{\efascio}=X\smallsetminus \{x_0\}$ otherwise. If $\sigma(U_{\efascio})\subseteq\isom(V\otimes\ox,\rest{\efascio^{\vee}}{U_{\efascio}})/G$ we will say that the singular principal Higgs $G$-bundle is \textbf{honest}.
\end{deff}
\begin{oss}\label{remark-det-banale}
If the singular principal Higgs $G$-bundle $(\efascio,\tau)$ is honest and $U_{\efascio}=X$ then $\det\efascio\simeq\ox$.
\end{oss}

\subsection{Descending principal Higgs bundles}

In this section we want to show that there exists a one-to-one correspondence between singular principal Higgs bundles $(\efascio,\tau,\phi)$ over $X$ and what we will call \textit{descending principal Higgs bundles} (over $\xtilde$).\\

Bhosle shows in \cite{bhosle1992generalised} that there is a one to one correspondence between torsion free sheaves over a nodal curve $X$ and \textit{generalized parabolic vector bundles} over the normalization $\xtilde$ of the nodal curve. We recall that a \textit{generalized parabolic vector bundle} or, for simplicity, a \textit{parabolic vector bundle} with support the divisor $D$, is a pair $(E,q)$ where $E$ is a vector bundle over $\xtilde$ and $q:E_D\to R$ is a surjective homomorphism of vector spaces. In our case we choose $D=x_1+x_2$ and therefore $q:E_{x_1}\oplus E_{x_2}\to R$ is a surjective morphism of vector spaces.
More precisely, if $(E,q)$ is a generalized parabolic vector bundle over $\xtilde$, Bhosle shows that the sheaf $\efascio$
\begin{equation}\label{eq-fascio-tor-free}
\efascio\doteqdot Ker[\nu_*E\longrightarrow\nu_*(E_{x_1}\oplus E_{x_2})\simeq E_{x_1}\oplus E_{x_2}\stackrel{q}{\longrightarrow}R]
\end{equation}
is a torsion free sheaf over $X$ such that $\nu^*\efascio=(E,q)$.\\

We recall that a torsion free sheaf $\efascio$ over a curve $Y$ is said to be \textbf{(semi)stable} if and only if for every non trivial subsheaf $\mathcal{F}\subset\efascio$ the following inequality holds:
\begin{equation*}
\frac{\deg\mathcal{F}}{\rk{\mathcal{F}}}(\leq)\frac{\deg\efascio}{\rk{E}}
\end{equation*}
where the degree of a torsion free sheaf is defined by the equality:
\begin{equation*}
\chi({\efascio})\doteqdot h^0(Y,\efascio)-h^1(Y,\efascio)=\deg\efascio+\rk{\efascio}(1-g).
\end{equation*}
The degree of a locally-free sheaf or of the corresponding vector bundle could also be defined as the degree of the determinant, i.e. the degree of the associated divisor.\\

On the other hand, a generalized parabolic vector bundle $(E,q)$ over a smooth curve is said to be \textbf{$\alpha$-(semi)stable} (for $\alpha\in[0,1]\cap\qq$) if and only if for every subbundle $F\subset E$ the following inequality holds:
\begin{equation*}
\frac{\degparalpha F}{\rk{F}}(\leq)\frac{\degparalpha E}{\rk{E}},
\end{equation*}
where the $\alpha$-parabolic degree of $F\subset E$ is defined as follows:
\begin{equation*}
\degparalpha F\doteqdot\deg F-\alpha\;\dim q(F_{x_1}\oplus F_{x_2}).
\end{equation*}

Bhosle shows in \cite{bhosle1992generalised} Proposition 1.9 that a torsion free sheaf $\efascio$ over $X$ is (semi)stable if and only if the corresponding generalized parabolic vector bundle $(E,q)$ over $\xtilde$ is $1$-(semi)stable. From now on we fix the stability parameter $\alpha=1$ and we will write $\degpar$ for $1$-$\degpar$.\\

Now let $(E,q,\tautilde,\phitilde)$ be a quadruple where $(E,q)$ is a generalized parabolic vector bundle over $\xtilde$ and $\tautilde:\symstar(E\otimes V)^G\to\oxtilde$, $\phitilde:E\to E\otimes\omegaxt$ are morphisms of $\oxtilde$-algebras and $\oxtilde$-modules respectively. We define the \textbf{push forward} $(\efascio,\tau,\phi)\doteqdot\nu_*(E,q,\tautilde,\phitilde)$ of such quadruples as follows: the torsion free sheaf $\efascio$ is the sheaf associated to the generalized parabolic bundle $E$ as in \eqref{eq-fascio-tor-free}, from the inclusion $\efascio\subset\nu_*E$ we get a morphism $\nu^*\efascio\to\nu^*\nu_*E\to E$, so we define:
\begin{align*}
\phitilde':\nu^*\efascio\longrightarrow E\stackrel{\phitilde}{\longrightarrow}E\otimes\omegaxt\\
\phi:\efascio\longrightarrow\nu_*\nu^*\efascio\stackrel{\nu_*\phitilde'}{\longrightarrow}\nu_*(E\otimes\omegaxt),
\end{align*}
and
\begin{align*}
\tautilde':\symstar(\nu^*\efascio\otimes V)^G\to\symstar(E\otimes V)^G\stackrel{\tautilde}{\longrightarrow}\oxtilde\\
\tau:\symstar(\efascio\otimes V)^G\longrightarrow\symstar(\nu_*\nu^*\efascio\otimes V)^G\stackrel{\nu_*\tautilde'}{\longrightarrow}\nu_*\oxtilde.
\end{align*}

Conversely let $(\efascio,\tau,\phi)$ be a singular principal Higgs $G$-bundle over $X$, then from the torsion free sheaf $\efascio$ we get a generalized parabolic vector bundle $(E,q)$ over $\xtilde$ defining $E\doteqdot\nu^*\efascio$ and $q:E_{x_1}\oplus E_{x_2}\to\text{Coker}(\efascio\to\nu_*\nu^*\efascio)$ (see \cite{schmitt2005singular}); moreover we set $\phitilde\doteqdot\nu^*\phi$ and $\tautilde\doteqdot\nu^*\tau$. So we get a quadruple $(E,q,\tautilde,\phitilde)$ such that $\nu_*(E,q,\tautilde,\phitilde)\simeq(\efascio,\tau,\phi)$.

\begin{oss}
From the inclusion $i\colon\efascio\otimes\omegax\to\nu_*E\otimes\nu_*\omegaxt$ and the exactness of the following sequence:
\begin{equation*}
0\to K_{x_0}\to\nu_*E\otimes\nu_*\omegaxt\to\nu_*(E\otimes\omegaxt)\to0
\end{equation*}
we get an inclusion $j\colon\efascio\otimes\omegax\to\nu_*(E\otimes\omegaxt)$, being $\efascio\otimes\omegax$ torsion free and $K_{x_0}$ of pure torsion.
\end{oss}

\begin{deff}[\textbf{Descending principal Higgs bundles}]\label{def-descending}
A descending principal Higgs $G$-bundle on $\xtilde$ is a quadruple $(E,q,\tautilde,\phitilde)$ where $(E,q)$ is a generalized parabolic vector bundle over $\xtilde$, $\tautilde:\symstar{(E\otimes W)}^G\to\oxtilde$ is a homomorphism of $\oxtilde$-algebras and $\phitilde:\xtilde\to \zaza{E}\otimes\Omega^1_{\xtilde}$ is a section such that:
\begin{enumerate}
\item The pair $(E,\tautilde)$ defines a principal $G$-bundle $\mathcal{P}(E,\tautilde)$ on $\xtilde$ (therefore $\det E\simeq\oxtilde$, see Remark \ref{remark-det-banale});
\item The image of the homomorphism $\tau$ from the triple $(\efascio,\tau,\phi)=\nu_*(E,q,\tautilde,\phitilde)$ lies in the sub algebra $\ox$ of $\nu_*\oxtilde$;\label{cond1-decor}
\item The image of the homomorphism $\phi$ from the triple $(\efascio,\tau,\phi)=\nu_*(E,q,\tautilde,\phitilde)$ lies in $\zaza{\efascio}\otimes\Omega^1_{X}$.\label{cond2-decor}
\end{enumerate}
If we do not require the conditions \eqref{cond1-decor} and \eqref{cond2-decor} we will call such a quadruple \textbf{singular Higgs $G$-bundle with a generalized parabolic structure (=GPS)}.
\end{deff}

\subsubsection*{Cuspidal case}
Let $X$ be a cuspidal curve, $\nu:\xtilde\to X$ the normalization map and $\{x_1\}=\nu^{-1}(x_0)$ the preimage of the cuspidal point $x_0\in X$. Bhosle showed in \cite{bhosle1996generalized} Proposition 4.7 that there is a correspondence between torsion free sheaves on $X$ and generalized parabolic bundles with parabolic structure on the divisor $2x_1$. Also in this case the (semi)stability notions for these objects coincide. Therefore one can easily extend our constructions to the cuspidal case. From now on we will work assuming $X$ is a nodal curve, but all results, with the obvious slight modifications, will also hold when $X$ has a cusp.

\section{Decoration and (semi)stability}\label{secdecoration-semistability}
\subsection{Singular Higgs $G$-Bundles with GPS as decorated bundles}\label{secsingular}

Let $(E,q,\tautilde,\phitilde)$ a singular Higgs $G$-bundle with a GPS, since $\symstar(E\otimes V)^G$ is a finitely generated $\oxtilde$-algebra, we get a surjective morphism
\begin{equation*}
\bigoplus_{i=1}^s(E\otimes V)^{\otimes i}\longrightarrow\symstar(E\otimes V)^G,
\end{equation*}
for some $s\in\nn$, so $\tautilde$ induces a map
\begin{equation*}
 \varphi_1'':\bigoplus_{i=1}^s(E\otimes V)^{\otimes i}\longrightarrow\symstar(E\otimes V)^G\stackrel{\tautilde}{\longrightarrow}\oxtilde.
\end{equation*}
Let $N_i=\dim(V^{\otimes i})$, then we get a morphism
\begin{equation*}
 \varphi_1':\bigoplus_{i=1}^sE^{\otimes iN_i}\longrightarrow\oxtilde.
\end{equation*}
Denote $E\abc{}\doteqdot (E^{\otimes a})^{\oplus b}\otimes (\det E)^{\otimes -c}$. Since the natural representation $Gl(V)\to Gl(\bigoplus_{i=1}^s V^{\otimes iN_i})$ is homogeneous, i.e. $\forall z\in\cc^*$ $z\cdot Id_{GL(V)}\mapsto z^{h}\cdot Id_{Gl(\oplus_{i=1}^s V^{\otimes iN_i})}$ for some integer $h$, by Corollary 1.1.5.4 in \cite{schmitt2008geometric} there exist integers $a_1,b_1,c_1$ such that
\begin{equation*}
E\abc{1}=\bigoplus_{i=1}^sE^{\otimes iN_i}\oplus W
\end{equation*}
for a suitable vector bundle $W$. Therefore we can extend $\phi'$ to a morphism
\begin{equation*}
\varphi_1:E\abc{1}\longrightarrow\oxtilde,
\end{equation*}
by setting $\rest{\varphi_1}{W}=0$ and $\rest{\varphi_1}{W^\perp}=\varphi_1'$ (see \cite{schmitt2005singular} Section 3 for details).\\
Conversely, if $E\abc{1}$ decomposes as $\bigoplus_{i=1}^sE^{\otimes iN_i}\oplus W$ and $\rest{\varphi_1}{W}\equiv 0$, then $\varphi_1$ induces a morphism $\tautilde:\symstar(E\otimes V)^G\to\oxtilde$.\\

Now we look at $\phitilde:\zaza{E}\to\omegaxt$. Since $\home(\oxtilde,\omegaxt)\simeq H^0(\xtilde,\omegaxt)$, if the genus of $\xtilde$ is $\geq 1$ we can choose a section $\omega:\det E\simeq\oxtilde\to\omegaxt$ not identically zero. Let $\rho':Gl(V)\to Gl(\zaza{V}\oplus\cc)$ be the natural representation obtained by identifying $\cc$ with $\bigwedge^{\dim V}V$. The pair $(\phitilde,\omega)$ induces a map 
\begin{equation*}
\varphi_2':E_{\rho'}=E\times_{\rho'}(\zaza{V}\oplus\cc)\to\omegaxt;
\end{equation*}
indeed $E_{\rho'}\simeq\zaza{E}\oplus\oxtilde$ and so $\varphi_2'(e,\mu)=\phitilde(e)+\omega(\mu)$ for any $e\in\zaza{E}$ and $\mu\in\oxtilde$ over the same point $x\in\xtilde$.
Since the representation $\rho'$ is homogeneous, as before there exist $a_2,b_2,c_2\in\mathbb{N}$ such that $\rho\abc{2}=\rho'\oplus\rho''$ for a suitable representation $\rho''$. Here $\rho\abc{2}:Gl(V)\to Gl(V\abc{2})$ is the obvious representation in $V\abc{2}=(V^{\otimes a_2})^{\oplus b_2}\otimes(\bigwedge^{\dim V}V)^{\otimes -c_2}$.
Therefore, since 
\begin{equation*}
E\abc{2}=(E^{\otimes a_2})^{\oplus b_2}\otimes(\bigwedge^{\rk E}E)^{\otimes -c_2}=E_{\rho'}\oplus W
\end{equation*}
for a suitable vector bundle $W$, the map $\varphi_2'$ extends to a map
\begin{equation*}
\varphi_2:E\abc{2}\longrightarrow\omegaxt
\end{equation*}
such that $\rest{\varphi_2}{W}\equiv 0$ and $\rest{\varphi_2}{E_{\rho'}}=\varphi_2'$. Conversely, if $\varphi_2$ and $\rho\abc{2}$ are defined as before they give rise non-zero morphisms $\phitilde:\zaza{E}\to\omegaxt$ and $\omega:\oxtilde\to\omegaxt$.
\begin{oss}
Suppose the genus of $\xtilde$ is strictly positive and fix a morphism $\omega\in\home(\oxtilde,\omegaxt)$ not identically zero. There is a natural inclusion of Higgs fields in a projective space given as follows:
\begin{align*}
\home(\zaza{E},\omegaxt) & \hookrightarrow\pp\left(\home(\zaza{E},\omegaxt)\oplus<\omega>\right)\\
v\quad &\hookrightarrow\quad [v:1]
\end{align*}
where with $<\omega>$ we denote the linear subspace of $\home(\oxtilde,\omegaxt)$ generated by $\omega$. Note that $[v:1]$ and $[\mu\cdot v:1]$ are different points for any $\mu\in\cc\smallsetminus\{1\}$.
\end{oss}
\begin{oss}
For any representation $\varrho:Gl(V)\to Gl(W)$ we can extend the previous construction  to $\varrho$-pairs $(E,\phi)$ where $E$ is a vector bundle with fiber $V$, $\phi:E_{\varrho}\to L$ is a morphism of vector bundles and $L$ is a line bundle of positive degree.
\end{oss}

So we define:
\begin{deff}[\textbf{Double-decorated generalized parabolic bundles}]
A double-decorated generalized parabolic bundle (= \ddgpb) with decoration of type $\tipo\doteqdot(d,r,a_1,b_1,c_1,a_2,b_2,c_2,L_1,L_2)$ is a quadruple $(E,q,\varphi_1,\varphi_2)$ where $(E,q)$ is a generalized parabolic bundle of rank $r$ and degree $d$, while
\begin{align*}
&\varphi_1:E\abc{1}\longrightarrow L_1\\
&\varphi_2:E\abc{2}\longrightarrow L_2,
\end{align*}
are morphisms called decorations. Sometimes we will call a \ddgpb just ``double-decorated''.
\end{deff}

\begin{oss}\label{oss-det-banale-per-discendenti}
\begin{enumerate}
\item If a \ddgpb $(E,q,\varphi_1,\varphi_2)$ is induced by a descending Higgs $G$-bundle, then $\det E\simeq\oxtilde$ and so $E\abc{}=(E^{\otimes a})^{\oplus b}\otimes(\det E)^{\otimes -c}\simeq E\ab{}$.
\item For what we said before a \ddgpb obviously generalize singular Higgs bundles with GPS.
\end{enumerate}
\end{oss}

\subsection{(Semi)stability for the double-decorated bundles}\label{secsemista}
In this section we want to define a notion of (semi)stability for a \ddgpb, so let $(E,q,\varphi_1,\varphi_2)$ be a \ddgpb with a decoration of type $(d,r,\underline{a},\underline{b},\underline{c},L_1,L_2)$, where $\underline{a},\underline{b},\underline{c}\in\nn\times\nn$.\\

A representation $\varrho:G\to Gl(V)$ gives rise to an action of $G$ on $V$ and so to an action on $\pp(V)$ with a linearization on the line bundle $M=\mathcal{O}_{\pp(V)}(1)$.
Now for any $x\in \pp(V)$, and any one-parameter subgroup $\lambda:\cc^*\to G$, the point $x_{\infty}= \lim_{z\to\infty} \lambda(z)\ldotp x$ is a fixed point for the action of $\cc^*$ induced by $\lambda$. So the linearization provides a linear action of $\cc^*$ on the one dimensional vector space $M_{x_{\infty}}$. This action is of the form $z\ldotp v=z^\gamma v$ for some $\gamma \in \mathbb Z$, and finally we define

\begin{equation*}
\mu_{\rho}(\lambda;x)=-\gamma.
\end{equation*}

More generally, if we have an action $\chi$ of an algebraic group $G$ on a projective variety $Y$ and a linearization of the action to a line bundle $M$, we can define in the same way $\mu_\chi(\lambda;y)$ for any one-parameter subgroup $\lambda$ and any $y\in Y$. 
Finally, if we have a morphism of projective varieties $\sigma:X\to Y$, we define 

\begin{equation*}
\mu_{\chi}(\lambda;\sigma)\doteqdot\max_{x\in X}\;\mu_{\chi}(\lambda;\sigma(x)).
\end{equation*}

Given a representation $\varrho:Gl(V)\to Gl(W)$, a nonzero map $\varphi:E_{\varrho}\to L$ provides a section $\sigma:\xtilde\to\pp(E_{\varrho})$ and vice versa.

So given a \ddgpb $(E,q,\varphi_1,\varphi_2)$, the maps $\varphi_i:E\abc{i}\to L_i$ provide sections $\sigma_i:\xtilde\to\pp(E\abc{i})$ for $i=1,2$.\\

Let $\lambda:\cc^*\to G$ be an one-parameter subgroup of $G$, or equivalently let $(\efil,\alphafil)$ be the corresponding weighted filtration of $E$, then we will denote by
\begin{equation*}
\mu_{\rho\abc{i}}(\lambda;\varphi_i)\equiv\mui\doteqdot\mu_{\rho\abc{i}}(\lambda;\sigma_i)
\end{equation*}
where, with some abuse of notation, we denote by $\lambda$ also the induced one-parameter subgroup $\cc^*\stackrel{\lambda}{\longrightarrow}G\stackrel{\rho}{\longrightarrow}Gl(V)\stackrel{\rho\abc{i}}{\longrightarrow} Gl(V\abc{i})$, obtained by composing $\lambda$ with the representations $\rho$ and $\rho\abc{i}$, and we denote by $\rho\abc{i}$ also the composition $\rho\abc{i}\circ\rho$.\\

Sometimes, for sake of convenience, we will write $\mu(F,E)$ instead of $\mu_{\rho_{a,b,c}}(0\subset F\subset E,(1);\varphi)$.

\begin{oss}
\begin{enumerate}
 \item $\mu_{\rho\abc{i}}(\lambda;\sigma_i(x))=\mu_{\rho\abc{i}}(\lambda;\sigma_i(y))$ for all $x,y$ belonging to the same irreducible component of $\xtilde$. (\cite{schmitt2000universal} Remark 1.5)
 \item For $i=1,2$ the following equality holds:
\begin{equation*}
\mu_{\rho\abc{i}}(\lambda;\sigma_i(x))=-\min_j\{\gamma_j^{(i)}\; | \; (\sigma_i(x))(v_j^{(i)})\neq 0\} 
\end{equation*}
where $\{v_j^{(i)}\}_j$ is a base of eigenvectors for the action of $\lambda$ over $V\abc{i}$ and, with some abuse of notation, by writing $(\sigma_i(x))(v_j^{(i)})$ we mean that we have chosen a representative of the class $\sigma_i(x)\in\pp(E\abc{i})$, and so we can think of $\sigma_i(x)$ as an element of $V\abc{i}^{\vee}$.
 \item The following equality holds for $i=1,2$:
\begin{equation*}
\muic=-\min\{\gamma_{j_1}+\dots+\gamma_{j_{a_i}}\; | \; \rest{\varphi_i}{(E_{j_1}\otimes\dots\otimes E_{j_{a_i}})^{\oplus b_i}}\not\equiv0\}
\end{equation*}
where $\efil:\;0\subset E_1\subset\dots\subset E_s\subset E_{s+1}=E$ is the weighted filtration with weights $\alphafil=(\alpha_j)_{j\leq s}$ induced by the one-parameter subgroup $\lambda$, while
\begin{equation*}
\underline{\gamma}=(\gamma_1,\dots,\gamma_r)\doteqdot\sum_{j=1}^s \alpha_j (\underbrace{\rk{E_j}-r,\dots,\rk{E_j}-r}_{\rk{E_j}\text{-times}},\underbrace{\rk{E_j},\dots,\rk{E_j}}_{r-\rk{E_j}\text{-times}}).
\end{equation*}
See \cite{schmitt2005singular} Remark 3.1.1.
 \item If the pair $(\phitilde,\omega)$ induces $\varphi_2$ as in Section \ref{secsingular} we have that
\begin{equation*}
\mu_{\rho\abc{2}}(\filtrazione;\varphi_2)=\mudus.
\end{equation*}
\end{enumerate}
\end{oss}

\begin{deff}[\textbf{(Semi)stable \ddgpb bundles}]\label{def-semistable-ddgpb}
We will say that the decorated bundle $(E,q,\varphi_1,\varphi_2)$ with a decoration $(d,r,\underline{a},\underline{b},\underline{c},L_1,L_2)$ is $(\delta_1,\delta_2)$-(semi)stable, for $\delta_i\in\qq_{>0}$, if for all weighted filtrations $(\filtrazione)$ the following inequality holds:
\begin{equation*}
P(\filtrazione)+\delta_1\,\munoc+\delta_2\,\muduc\geq0
\end{equation*}
where
\begin{equation*}
P(\filtrazione)\doteqdot\sum_{j=1}^s\alpha_j[\rk{E_j}\degpar(E)-\rk{E}\degpar(E_j)].
\end{equation*}
\end{deff}

\subsection{$\beta$-filtrations and Higgs sections}

In this section we will recall the notion of $\beta$-filtration (\cite{schmitt2005singular} pg. 217) and Higgs section, or equivalently Higgs reduction, (Definition \ref{def-higgs-reduction}). In the next section, thanks to these notions, we will be able to define (semi)stability for descending principal Higgs bundles (Definition \ref{def-semistable-descending-higgs-bundles}) and at last we will show the equivalence between such definition and the Definition \ref{def-semistable-ddgpb}.\\

Let $\rho:G\to Sl(V)\hookrightarrow Gl(V)$ the faithful representation fixed at the beginning, so we can identify $G$ with a subgroup of $Sl(V)$. Given a one-parameter subgroup $\lambda:\cc^*\to G$ we denote by
\begin{equation*}
Q_G(\lambda)\doteqdot \{g\in G \; | \; \exists \lim_{z\to\infty}\lambda(z)\cdot g\cdot \lambda(z)^{-1}\}
\end{equation*}
the parabolic subgroup of $G$ induced by $\lambda$.\\

Before explaining what we mean by $\beta$-filtration we recall some general results about parabolic subgroups and representations theory.
\begin{oss}\label{oss-beta-filtr}
\begin{enumerate}
\item Since $G$ is reductive, if $Q'$ is a parabolic subgroup of $Gl(V)$, then $Q'\cap G$ is a parabolic subgroup of $G$.\label{oss1-beta-filtr} [Sketch: if $B$ ($B_G$) is a borelian subgroup of $Gl(V)$ (resp. of $G$) then, up to conjugacy class, $B\cap G=B_G$ and so $Q'\cap G\supseteq B_G$]
\item Given a parabolic subgroup $Q$ of $G$ and a representation $\rho$, we can construct a parabolic subgroup of $Gl(V)$; in fact, given $Q$, there exists a one-parameter subgroup $\lambda:\cc^*\to G$ such that $Q=Q_G(\lambda)$, then the set $Q_{Gl(V)}(\rho\circ\lambda)$ is a parabolic subgroup of $Gl(V)$.\label{oss2-beta-filtr}
\item Given $\lambda':\cc^*\to Gl(V)$, or equivalently the parabolic subgroup $Q'$ associated to $\lambda'$, there always exists $\lambda:\cc^*\to G$ such that $Q'\cap G=Q_G(\lambda)$ (see \eqref{oss1-beta-filtr}). The following diagram is \textbf{not} in general commutative:
\begin{equation*}
\xymatrix{ \cc^*\ar[d]_{\lambda}\ar[r]^{\lambda'} & Gl(V)\\
G\ar[ur]_{\rho}
}
\end{equation*}
\item Given a parabolic subgroup $Q'\subset Gl(V)$ and fixing a representation $\rho:G\to G(V)$, it is possible to define a parabolic subgroup $Q=Q'\cap G\subset G$ (see \eqref{oss1-beta-filtr}) and from $Q$ we can obtain a parabolic subgroup $Q''\subset Gl(V)$ as explained in \eqref{oss2-beta-filtr}. Therefore, fixing a base of $Gl(V)$, we have a map
\begin{equation*}
\zeta:\{\text{Parabolic subgroups of }Gl(V)\}\to\{\text{Parabolic subgroups of }Gl(V)\}.
\end{equation*}
\item We will call \textbf{stable} the parabolic subgroups of $Gl(V)$ such that $Q'=\zeta(Q')$, with respect to the same base of $Gl(V)$.
\end{enumerate}
\end{oss}

Now we want to construct a $\beta$-filtration of a singular Higgs $G$-bundle with GPS $(E,q,\tautilde,\phitilde)$ from a given one-parameter subgroup $\lambda:\cc^*\to G$ and a section $\beta:\xtilde\to\mathcal{P}(E,\tautilde)/Q_G(\lambda)$. Consider the principal $Q_G(\lambda)$-bundle $\beta^*\mathcal{P}(E,\tautilde)$. We define
\begin{equation*}
E_i^{\vee}\doteqdot\beta^*\mathcal{P}(E,\tautilde)\times_{\rho}V_i^{\vee}, \qquad \text{for }i=1,\dots,s,
\end{equation*}
this gives a filtration of $E^{\vee}$. Dualizing the inclusions $E_i^{\vee}\subset E^{\vee}$ and defining $E_i=\ker(E^{\vee}\to E_{s+1-i}^{\vee})$ we get a filtration $\efil_\beta$ of $E$. Moreover, setting $\alpha_i\doteqdot(\gamma_{i+1}-\gamma_{i})/r$ (where $\gamma_i$ are the weights related to $\lambda$) and $\alphafil_\beta\doteqdot(\alpha_s,\dots,\alpha_1)$, we get the desired weighted filtration $(\efil_{\beta},\alphafil_{\beta})$.\\

Conversely let $(\filtrazione)$ be a weighted filtration of $E$ and let $({\efil}^{\vee},\alphafil^{\vee})$ the corresponding weighted filtration of $E^{\vee}$ where $\alphafil^{\vee}_\beta=(\alpha_s,\dots,\alpha_1)$ if $\alphafil=(\alpha_1,\dots,\alpha_s)$. To this filtration one can associate the morphisms:
\begin{align*}
& \lambda':\cc^*\to Gl(V)\\
& \beta':\xtilde\to\isom(V\otimes\oxtilde,E^{\vee})/Q',
\end{align*}
where $Q'\doteqdot Q_{Gl(V)}(\lambda')$. Indeed the filtration $(\filtrazione)$ induces a weighted flag of $V$ and so an one-parameter subgroup of $Gl(V)$. Moreover the inclusion of principal bundles induces a section $\beta'$ as follows:
\begin{equation*}
\xymatrix{ \isom(V^{\bullet}\otimes\oxtilde,(\efil)^{\vee})\; \ar@{^{(}->}[r] \ar[d] & \isom(V\otimes\oxtilde,E^{\vee})\ar[d]\\
\xtilde \ar@{-->}[r]^(0.3){\beta'} & \isom(V\otimes\oxtilde,E^{\vee})/Q'.
}
\end{equation*}

Then we will say that $\efil$ is a \textbf{$\beta$-filtration}, and we will write $\efil_{\beta}$ (instead of $\efil$), if there exists $\beta:\xtilde\to\mathcal{P}(E,\tautilde)/Q$ such that the following diagram commutes:
\begin{equation*}
\xymatrix{ & \mathcal{P}(E,\tautilde) \ar@{^{(}->}[r] \ar[d] & \isom(V\otimes\oxtilde,E^{\vee}) \ar[d] \\
 & \mathcal{P}(E,\tautilde)/Q \ar@{^{(}->}[r]^(0.4){\overline{i}} & \isom(V\otimes\oxtilde,E^{\vee})/Q'\\
\xtilde \ar@{-->}[ur]^{\beta} \ar@/_1pc/[urr]_{\beta'}
}
\end{equation*}
where $Q=Q'\cap G$ and $\mathcal{P}(E,\tautilde)$ is the principal bundle over $\xtilde$ associated to $(E,\tautilde)$.

\begin{prop}
Let $\efil$ be filtration of $E$. Then $\efil$ is a $\beta$-filtration if and only if the parabolic subgroup $Q'$ associated to such filtration is stable (in the sense of Remark \ref{oss-beta-filtr} point (5)).
\end{prop}
\begin{proof}
The filtration $\efil$ gives rise to a subbundle $I_{Q'}$ of $\isom(V\otimes\oxtilde,E^{\vee})$, and so the inclusion $I_{Q'}\hookrightarrow\isom(V\otimes\oxtilde,E^{\vee})$ induces a section $\beta':\xtilde\to\isom(V\otimes\oxtilde,E^{\vee})/Q'$. We consider now the groups $Q$ and $Q''$ constructed as in Remark \ref{oss-beta-filtr} point (4) and the one-parameter subgroups $\lambda,\lambda'$ (respectively) associated to $Q$ and $Q'$. The following diagram commutes:
\begin{equation}\label{prop-beta-filtr-diagrammino}
\xymatrix{ \cc^*\ar[d]_{\lambda}\ar[r]^{\lambda'} & Gl(V)\\
G\ar[ur]_{\rho}
}
\end{equation}
in fact by hypothesis $Q'$ is stable and so $Q'=Q''$.\\
Consider now the following diagram:
\begin{equation*}
\xymatrix{ & \mathcal{P}(E,\tautilde) \ar@{^{(}->}[r]^(0.4){i} \ar[d] & \isom(V\otimes\oxtilde,E^{\vee}) \ar[d] \\
 & \mathcal{P}(E,\tautilde)/Q \ar@{^{(}-->}[r]^(0.4){\overline{i}} & \isom(V\otimes\oxtilde,E^{\vee})/Q'\\
\xtilde \ar@/_1pc/[urr]_{\beta'}
}
\end{equation*}
Note that the map $\overline{i}$ is well defined. In fact, denoting by $[\;\cdot\;]_{Q'}$ the class modulo $Q'$, the map $i$ induces $\overline{i}$ if and only if for any $q\in Q$ and for any $a'=q\ldotp a$ one has $[i(a')]_{Q'}=[i(a)]_{Q'}$. Since $i(a')=i(q\ldotp a)=\rho(q)\ldotp i(a)$, $[i(a')]_{Q'}=[i(a)]_{Q'}\Longleftrightarrow\rho(q)\in Q'$ but since $\rho(Q)\subseteq Q''=Q'$ we are done.\\

Since $\rho(Q)\subseteq Q''$ stabilize the filtration $\efil$ we can consider the subbundle $I_Q$ of $\mathcal{P}(E,\tautilde)\hookrightarrow\isom(V\otimes\oxtilde,E^{\vee})$, for which $I_Q\times_{\rho}Q''=I_{Q'}$. The inclusion $I_Q\hookrightarrow\mathcal{P}(E,\tautilde)$ induces a morphism $\beta:\xtilde\to\mathcal{P}(E,\tautilde)/Q$ that makes the following diagram commute:
\begin{equation}\label{prop-beta-filtr-diagrammone}
\xymatrix{I_{Q'} \ar@/^1pc/[drrr] \ar@/_1pc/[dddr]\\
& I_Q \ar[ul] \ar@{^{(}->}[r] \ar[dd]& \mathcal{P}(E,\tautilde) \ar@{^{(}->}[r]^{i} \ar[d] & \isom(V\otimes\oxtilde,E^{\vee}) \ar[d] \\
& & \mathcal{P}(E,\tautilde)/Q \ar@{^{(}-->}[r]^{\overline{i}} & \isom(V\otimes\oxtilde,E^{\vee})/Q'\\
& \xtilde \ar[ur]^{\beta} \ar@/_1pc/[urr]_{\beta'}
}
\end{equation}
and so we are done. Equivalently, we could show that $\text{Im}(\beta')\cap\text{Im}(\overline{i})=\text{Im}(\overline{i})$ and define $\beta=\overline{i}^{-1}\circ\beta'$.\\

The $\Rightarrow$ arrow is obvious.
\end{proof}

\begin{oss}
In the previous proposition we have shown that the following conditions are equivalent:
\begin{enumerate}
\item $\efil=\efil_{\beta}$.
\item The diagram \eqref{prop-beta-filtr-diagrammino}commutes
\item The diagram \eqref{prop-beta-filtr-diagrammone} commutes.
\item $Q'$ is stable.\\
\end{enumerate}
\end{oss}
     %
\begin{oss}\label{prop-not-max}
Let $G$ be a semisimple group, $\rho:G\to Sl(V)\subset Gl(V)$ a faithful representation and $\lambda:\cc^*\to G$ a one-parameter subgroup such that $Q_{G}(\lambda)$ is a maximal parabolic subgroup of $G$. Then the parabolic subgroup $Q_{Gl(V)}(\rho\circ\lambda)$ of $Gl(V)$ is \textbf{not} maximal. Thanks to this observation we obtain that every $\beta$-filtration $\efil_{\beta}: \; 0\subset E_1\subset\dots\subset E_s\subset E_{s+1}=E$ has lenght greater or equal than $2$, i.e. $s\geq2$.
\end{oss}

Notice that the previous remark tells us that the parabolic subgroup of $G$ associated to a $\beta$-filtration is always a proper subgroup. Therefore, according to the definition of Ramanathan, the (semi)stability condition is checked only for maximal \textit{proper} parabolic subgroups of $G$. If $G$ is reductive but not semisimple Remark \ref{prop-not-max} does not hold in general as the following example shows:

\begin{esempio}
Consider $G=Gl(k)$, $\rho:Gl(k)\to Gl(n)$ ($k<n$) the inclusion (in the left up corner) and $\lambda:\cc^*\to G$ given by
\begin{equation*}
 \lambda(z)\doteqdot\left( \begin{array}{cccc}
 z^{\gamma} & 0 & \dots & 0 \\
 0 & 1 & \ddots & \vdots \\
 \vdots & \ddots & \ddots & 0\\
 0 & \dots & 0 & 1\\
\end{array}\right).
\end{equation*}
Therefore $\lambda'=\rho\circ\lambda$ is given by
\begin{equation*}
 \lambda'(z)\doteqdot\left( \begin{array}{ccccccc}
 z^{\gamma} & 0 & \dots & & & \dots & 0\\
 0 & 1 & \ddots & & & & \vdots\\
 \vdots & \ddots & \ddots & & & & \\
  & & & 1 & & & \\
  & & & & 1 & \ddots & \vdots \\
\vdots & & & & \ddots & \ddots & 0\\
0 & \dots & & & \dots & 0 & 1
\end{array}\right).
\end{equation*}
give arise to a maximal parabolic subgroup.
\end{esempio}

If $(\efascio,\tau,\phi)$ is a honest singular principal $G$-bundle on $X$ there is an analogous definition of $\beta$ filtration (see \cite{schmitt2005singular} for more details).\\

We will now recall the notion of Higgs reduction (\cite{bruzzo2010semistable} Definition 2.3) and then we will define (semi)stability for descending principal Higgs bundles and for honest singular principal Higgs $G$-bundles and we will proof that they are equivalent.\\
Let $(E,q,\tautilde,\phitilde)$ be a descending principal Higgs bundle, $K$ a closed subgroup of $G$ and  $\beta:\xtilde\to\mathcal{P}(E,\tautilde)/K$ a reduction of the structure group to $K$. The principal $K$-bundle $\mathcal{P}(E,\tautilde)_{\beta}=\beta^*(\mathcal{P}(E,\tautilde))$ on $\xtilde$ and the principal bundle injection $i_{\beta}:\mathcal{P}(E,\tautilde)_{\beta}\to\mathcal{P}(E,\tautilde)$ induce an injective morphism of bundles $\text{Ad}(\mathcal{P}(E,\tautilde)_{\beta})\to\text{Ad}(\mathcal{P}(E,\tautilde))$. Let $\Pi_{\beta}:\text{Ad}(\mathcal{P}(E,\tautilde)\otimes\omegaxt\longrightarrow(\text{Ad}(\mathcal{P}(E,\tautilde)/\text{Ad}(\mathcal{P}(E,\tautilde)_{\beta})\otimes\omegaxt$ be the induced projection.

\begin{deff}[\textbf{Higgs reduction}]\label{def-higgs-reduction}
A section $\beta:\xtilde\to\mathcal{P}(E,\tautilde)/K$ is a Higgs reduction of $(E,q,\tautilde,\phitilde)$ if $\phitilde\in\textnormal{ker}\,\Pi_{\beta}$.
\end{deff}

Observe that if $\beta$ is a Higgs reduction and $K$ is a parabolic subgroup of $G$ the filtration $(\efil_{\beta},\alphafil_{\beta})$ is $\phitilde$-invariant, i.e., $\phitilde(E_i)\subseteq E_i$ for all indices $i$.

\subsection{Equivalence among the semistability conditions} \label{sec-equiv-of-conditions-of-semistability}

We have started with principal Higgs bundles over a nodal curve $X$ and have generalized them to (honest) singular principal Higgs bundles over $X$. For the latter there is a natural condition of (semi)stability (Definition \ref{def-semistable-honest-sing-princ-higgs-bundles}). Then we have seen that dealing with such objects is the same as dealing  with descending principal Higgs bundles over the normalization $\xtilde$ of the curve $X$. We have shown that the descending principal Higgs bundles are a special case of the \ddgpb. For all such objects one has a notion of (semi)stability, see Definitions \ref{def-semistable-ddgpb} and \ref{def-semistable-descending-higgs-bundles}. In this section we show that the previous definitions of (semi)stability are equivalent, see Proposition \ref{prop-equiv-semista-honest-singular-vs-descending} and Theorem \ref{teo-semistable-descending-vs-ddgpb}.\\
If the curve $X$ is smooth the definition of (semi)stability for honest singular principal Higgs bundles extends the classical notion of (semi)stability of principal Higgs bundles and so we give a generalization for this definition in the nodal case.\\
On the other hand, if $G=Sl(V)$, a descending principal Higgs $G$-bundle over $\xtilde$ (\cite{schmitt2005singular} pg. 218) corresponds to a generalized parabolic Higgs vector bundles over $\xtilde$ and the latter corresponds to a torsion free sheaf over $X$ with a Higgs field. Bhosle in \cite{bhosle1992generalised} gives a notion of (semi)stability for parabolic vector bundles over $\xtilde$ which we generalize to the case of parabolic Higgs vector bundles. Therefore we need to show that the two notions of (semi)stability are, in this special case, the same (Remark \ref{oss-equiv-defini-semis-per-G=Sl}). Moreover we also show that the Definition \ref{def-semistable-descending-higgs-bundles} is equivalent to the definition of (semi)stability for torsion free sheaves with a Higgs field (see Proposition \ref{prop-equiv-semist-decorated-e-torsion-free}).\\

\begin{deff}[\textbf{(Semi)stable honest singular principal Higgs $G$-bundles}]\label{def-semistable-honest-sing-princ-higgs-bundles}
A honest singular principal Higgs $G$-bundle $(\efascio,\tau,\phi)$ over $X$ is (semi)stable if and only if
\begin{equation*}
L(\efascio^{\bullet}_{\beta},\underline{\alpha}_{\beta})\doteqdot\sum_{i=1}^s \alpha_i \left(\deg\efascio\;\rk{\efascio_i}-\deg\efascio_i\;\rk{\efascio}\right)(\geq)0
\end{equation*}
for every $\phi$-invariant weighted $\beta$-filtration $(\efascio^{\bullet}_{\beta},\underline{\alpha}_{\beta}):\quad 0\subset\efascio_1\subset\dots\subset\efascio_s\subset\efascio$.
\end{deff}
\begin{oss}
In the previous definition, instead of requiring that $L(\efascio^{\bullet}_{\beta},\underline{\alpha}_{\beta})(\geq)0$ for every $\phi$-invariant weighted filtration $(\efascio^{\bullet}_{\beta},\underline{\alpha}_{\beta})$, we could require the same inequality holds for every one-parameter subgroup $\lambda:\cc^*\to G$ and for every Higgs section $\beta:X\to\mathcal{P}(X,\tau)/Q_G(\lambda)$.
\end{oss}

\begin{deff}[\textbf{(Semi)stable descending principal Higgs $G$-bundles}]\label{def-semistable-descending-higgs-bundles}
Let $\mathfrak{E}=(E,q,\tautilde,\phitilde)$ be a descending principal Higgs $G$-bundle over $\xtilde$. We say that $\mathfrak{E}$ is (semi)stable if and only if for all $\lambda:\cc^*\to G$ and for all Higgs-section $\beta:\xtilde\to\mathcal{P}(E,\tautilde)/Q_G(\lambda)$ the following inequality holds:
\begin{equation*}
P(\efil_{\beta},\alphafil_{\beta})\doteqdot\sum_{i=1}^{s}\alpha_i(\rk{E_i}\degpar(E)-\rk{E}\degpar(E_i))\geq 0.
\end{equation*}
\end{deff}


\begin{deff}[\textbf{(Semi)stable generalized parabolic Higgs vector bundles}]\label{def-semistable-par-higgs-bundles}
A generalized parabolic Higgs vector bundles $(E,q,\phitilde)$ over $\xtilde$ is (semi)stable if for all subsheaves $F$ of $E$ such that $\rest{\phi}{F}:F\to F\otimes\omegax$ the condition:
$$
\frac{\degpar(F)}{\rk{F}}\doteqdot\mupar(F)\leq\mupar(E)\doteqdot\frac{\degpar(E)}{\rk{E}}
$$
holds.
\end{deff}

\begin{oss}\label{oss-equiv-defini-semis-per-G=Sl}
If $G=Sl(V)$ then Definitions \ref{def-semistable-descending-higgs-bundles} and \ref{def-semistable-par-higgs-bundles} are equivalent, indeed if $G=Sl(V)$ then all filtrations are $\beta$-filtrations and requiring that $\beta$ is a Higgs-section is the same as requiring that the filtration is $\phitilde$-invariant.
\end{oss}

\begin{prop}\label{prop-equiv-semist-decorated-e-torsion-free}
A parabolic Higgs vector bundle $\mathfrak{H}=(E,q,\phitilde)$ over $\xtilde$ is (semi)stable if and only if the corresponding Higgs torsion free sheaf $(\efascio,\phi)$ on $X$ is (semi)stable.
\end{prop}
\begin{proof}
We already now that the notion of (semi)stability for parabolic vector bundles over $\xtilde$ is equivalent to the (semi)stability for the associated torsion free sheaf over $X$ (\cite{bhosle1992generalised} Proposition 1.9). It remains to show that $F\subset E$ is $\phitilde$-invariant if and only if the corresponding torsion free $\mathcal{F}$ over $X$ is $\phi$-invariant.
Let us suppose that $F$ is $\phitilde$-invariant, i.e. $\rest{\phitilde}{F}:F\to F\otimes\omegaxt$. Recalling that the inclusion $\mathcal{F}\hookrightarrow\nu_*F$ gives an injective morphism $\nu^*\mathcal{F}\hookrightarrow F$, we can consider
\begin{equation*}
\rest{\phi}{\mathcal{F}}:\mathcal{F}\longrightarrow\nu_*\nu^*\mathcal{F}\stackrel{\nu_*(\phi'\mid_{\mathcal{F}})}{\longrightarrow}\nu_*F\otimes\nu_*\omegaxt.
\end{equation*}
Observing that $\nu_*F\cap\mathcal{E}=\mathcal{F}$ we are done.\\
Conversely given a $\phi$-invariant subsheaf $\mathcal{F}$, since $\phitilde=\nu^*\phi$, $E=\nu^*\efascio$ and $F=\nu^*\mathcal{F}\subseteq\nu^*\efascio=E$, we obtain that:
\begin{equation*}
\rest{\phitilde}{F}=\nu^*\rest{\phi}{\nu^*\mathcal{F}}:\nu^*\mathcal{F}\to\nu^*\mathcal{F}\otimes\nu^*\omegax,
\end{equation*}
and we are done.
\end{proof}

\begin{prop}[\textbf{Equivalence between Definitions \ref{def-semistable-honest-sing-princ-higgs-bundles} and \ref{def-semistable-descending-higgs-bundles}}]\label{prop-equiv-semista-honest-singular-vs-descending}
A honest singular principal Higgs bundle $(\efascio,\tau,\phi)$ over $X$ is (semi)stable if and only if the corresponding descending principal Higgs bundle $(E,q,\tautilde,\phitilde)$ over $\xtilde$ is (semi)stable.
\end{prop}
\begin{proof}
Observing that a $\beta$-filtration $(\efascio^{\bullet}_{\beta},\underline{\alpha}_{\beta})$ of $\efascio$ corresponds to a $\beta$-filtration $(\efil_{\beta},\alphafil_{\beta})$ of $E$, the results follows immediately from Proposition \ref{prop-equiv-semist-decorated-e-torsion-free}.
\end{proof}

We recall that a \ddgpb $(E,q,\varphi_1,\varphi_2)$ is $(\delta_1,\delta_2)$-semistable if for any weighted filtration $(\filtrazione)$ of $E$ we have:

\begin{equation*}
P(\filtrazione)+\delta_1\, \munoc+\delta_2 \,\muduc\geq 0.
\end{equation*}

A flat family $\mathfrak{F}$ of isomorphism classes of vector bundles on $\xtilde$ of type $(d,r)$ is said to be \textbf{bounded} if there exists a scheme $S$ of finite type over $\cc$ and a vector bundle $E_S$ on $S\times \xtilde$ such that for every vector bundle $E$ on $\xtilde$ with $E\in\mathfrak{F}$, there exists a point $s\in S$ with $E\simeq E_S|_{\{s\}\times \xtilde}$. 

\begin{prop}
A family $\mathfrak{F}$ of isomorphism classes of vector bundles of type $(d,r)$ is bounded if and only if there exists a constant $C$ such that for any $E\in \mathfrak{F}$ we have $\mu(E')\leq \frac{d}{r}+C$ for any subbundle $E' \subset E$.
\end{prop}

\begin{prop}
The family of $(\delta_1,\delta_2)$-semistable double-decorated generalized parabolic vector bundles $(E,q,\varphi_1,\varphi_2)$ of type $(d,r,\underline{a},\underline{b},\underline{c},L_1,L_2)$ is bounded.
\end{prop}
\begin{proof}
We know (\cite{schmitt2000universal} Lemma 1.8) that for a generic morphism $\varphi:E_{a,b}\to L$ we have 
$|\mu_{\rho_{a,b}}(F,E)|\leq a(r-1)$ for any subbundle $F\subset E$, and so, due to the semistability of $E$,
\begin{equation*}
P(F,E)\geq -(a_1(r-1))-(a_2(r-1))\doteqdot C.
\end{equation*}
Then, recalling that $\degpar(F)=\deg(F)-\dim q(F_{N_1}\oplus F_{N_2})$ we have
\begin{align*}
d\,\rk{F}-r(\deg(F)-r)&\geq(d-r)\rk{F}-r(\deg(F)-\dim q(F_{N_1}\oplus F_{N_2}))\nonumber\\
                 &=\degpar(E)\rk{F}-\degpar(F)\rk{E}\nonumber\\
                 &=P(F,E)\geq C
\end{align*}
and so
\begin{equation*}
\mu(F)\leq \frac{d+r^2-C}{r}.
\end{equation*}
\end{proof}


We saw that the family of descending principal Higgs bundles is contained in the family of double decorated vector bundles with trivial determinant. Now we want to relate the semistability concepts corresponding to the two families. Before that we need some preliminary results.

\begin{lem}\label{lemu}
Given a parabolic vector bundle $(E,q)$ with trivial determinant and morphisms $\varphi_1:E_{a_1,b_1}\to \oxtilde$, $\varphi_2:E_{a_2,b_2}\to\omegaxt$ induced respectively by morphisms $\tautilde:\symstar(E\otimes V)^G\to\oxtilde$, $\phitilde:E\to E\otimes\omegaxt$ and $\omega:\oxtilde\to\omegaxt$ as in Section \ref{secsingular}, one has
\begin{equation*}
\muno=0 \Leftrightarrow \mbox{$(\filtrazione)$ is a $\beta$-filtration}
\end{equation*}
and 
\begin{equation*}
\mudu=0 \Leftrightarrow \mbox{$(\filtrazione)$ is $\phitilde$-invariant}\\ \mbox{ i.e. $\phitilde(E_i)\subset E_i$ for any $i$}.
\end{equation*}
\end{lem}
\begin{proof}
See \cite{schmitt2005singular} Proposition 4.2.2 for the first equivalence and \cite{schmitt2000universal} Section 3.6 for the second.
\end{proof}

\begin{prop}\label{prhiggs}
The family of semistable descending principal Higgs bundles is bounded.
\end{prop}
\begin{proof}
Since the family of semistable Higgs vector bundles is bounded (see \cite{nitsure1991moduli}), then, following the idea of the proof of Proposition 4.12 in \cite{ramanathan1996moduli-1}, one easily sees that the family of semistable Principal Higgs $G$-bundles is bounded when $G$ is semisimple.
\end{proof}

\begin{teo}[\textbf{Equivalence between Definition \ref{def-semistable-ddgpb} and \ref{def-semistable-descending-higgs-bundles}}]\label{teo-semistable-descending-vs-ddgpb}
Given a descending principal Higgs bundle and a nonzero section $\omega:\oxtilde\to \omegaxt$ there exists $\delta$ such that for any $\delta_1,\delta_2\geq \delta$ the following conditions are equivalent:
\begin{enumerate}
\item[i)]
For any $\phitilde$-invariant $\beta$-filtration $(E^\bullet_{\beta},\underline{\alpha})$ one has

\begin{equation*}
P(E^\bullet_{\beta},\underline{\alpha})\geq 0
\end{equation*}

\item[ii)]
For any filtration $(\filtrazione)$, 
\begin{equation}\label{semistabilityddgpb}
P(\filtrazione)+\delta_1\muno+\delta_2\mudu\geq 0,
\end{equation}
where $\varphi_1$ and $\varphi_2$ are defined as in Lemma \ref{lemu}.
\end{enumerate}
\end{teo}
\begin{proof}
Let $(0\subset F\subset E,(1))$ be a $\phitilde$-invariant $\beta$-filtration. By Lemma \ref{lemu} 
\begin{equation*}
\muno=\mudu=0,
\end{equation*}
so $P(0\subset F\subset E,(1))\geq 0$.\\

Conversely, by Proposition \ref{prhiggs} there exists a constant $C$ such that $\mu(F)\leq \mu(E) +C$ for any $F\subset E$, and so for any weighted filtration $(\filtrazione)$,
\begin{equation*}
 P(\filtrazione)=\sum_{i=1}^s \alpha_i(\degpar(E)\,\rk{E_i}-r\,\degpar(E_i))\geq -\alpha r(r-1)C,
\end{equation*}
where $\alpha$ is $\max\{\alpha_i|i=1\dots s\}$. If the filtration is a $\phitilde$-invariant $\beta$-filtration then we are done. Otherwise, if the filtration is not $\phitilde$-invariant, one has $\mudu\geq r\alpha$, and if it is not a $\beta$-filtration, one has $\muno\geq 1$ (\cite{schmitt2000universal} Lemma 3.14). So if we choose $\delta=\max\{-Cr\alpha(r-1),-C(r-1)\}$ we obtain
\begin{equation*}
\begin{aligned}  P(\filtrazione)& +\delta_1\muno+\delta_2\mudu\\
                & \geq -\alpha r (r-1)C +\delta_1 \epsilon_1(\filtrazione) +\delta_2 \epsilon_2(\filtrazione)\geq 0
\end{aligned}
\end{equation*}
where 
\begin{equation*}
\epsilon_1(\filtrazione)=\begin{cases}  0 \quad &\mbox{if $(\filtrazione)$ is a $\beta$-filtration}\\
1 \quad &\mbox{otherwise} 
\end{cases}
\end{equation*}
and 
\begin{equation*}
\epsilon_2(\filtrazione)=\begin{cases}  0 \quad &\mbox{if $(\filtrazione)$ is a $\phitilde$-invariant}\\
1 \quad &\mbox{otherwise.} 
\end{cases}
\end{equation*}
Since $(\filtrazione)$ is not a $\phitilde$-invariant $\beta$-filtration, $\epsilon_1$ and $\epsilon_2$ cannot both be zero, the inequality \eqref{semistabilityddgpb} holds.
\end{proof}

\subsection{Families of \ddgpb}\label{sec-famiglie}

We want to define the concept of family of \ddgpb. We start with the notion of isomorphism between two \ddgpb.
\begin{deff}[\textbf{Isomorphism between \ddgpb}]
Let $(E,q,\varphi_1,\varphi_2)$ and \newline $(E',q',\varphi_1',\varphi_2')$ be double decorated generalized parabolic bundles with decoration of type $(d,r,\underline{a},\underline{b},\underline{c},L_1,L_2)$. They are isomorphic if and only if there exists an isomorphism of vector bundles $f:E\to E'$ such that the following diagrams commute:
\begin{equation*}\xymatrix{E_{x_1}\oplus E_{x_2} \ar[d]_{f}\ar[r]^(0.6){q} & R & \qquad E\abc{i}\ar[d]_{f\abc{i}}\ar[r]^(0.6){\varphi_i} & L_i\\
E'_{f(x_1)}\oplus E'_{f(x_2)} \ar[ur]_{q'} & \qquad & E'\abc{i} \ar[ur]_{\varphi_i'}
}
\end{equation*}
for $i=1,2$.
\end{deff}
\begin{deff}[\textbf{Family of \ddgpb}]
A family of double decorated generalized parabolic bundles of type $(d,r,\underline{a},\underline{b},\underline{c},L_1,L_2)$ parametrized by a scheme $S$ is a quadruple $(E_S,q_S,\varphi_{1_S},\varphi_{2_S})$ such that
\begin{itemize}
 \item[-] $E_S$ is a vector bundle over $\xtilde\times S$;
 \item[-] $q_S:\overline{\pi}_{S*}\rest{(E_S)}{\{x_1,x_2\}\times S}\longrightarrow R_S$, where $\overline{\pi}_{S}:\{x_1,x_2\}\times S\longrightarrow \{x_0\}\times S$ and $R_S$ is a vector bundle over $S$ of rank $r$;
 \item[-] $\varphi_{i_S}: (E_S)\abc{i}\to \pi_{\xtilde}^*L_i$ is a homomorphism such that $\rest{\varphi_{i_S}}{\{s\}\times\xtilde}\not\equiv0$ for $i=1,2$.
\end{itemize}
Moreover the pair $(E_S,q_S)$ is called a family of generalized parabolic vector bundles parametrized by $S$. For more details see \cite{schmitt2005singular} Section 2.3.\\

We will say that two families $(E_S,q_S,\varphi_{1_S},\varphi_{2_S})$ and $(E_S',q_S',\varphi_{1_S}',\varphi_{2_S}')$ are isomorphic if there exists an isomorphism of vector bundles $f_S:E_S\to E_S'$ such that the following diagrams commute:
\begin{equation*}\xymatrix{\overline{\pi}_{S*}(E_S)_{\{x_1,x_2\}\times S} \ar[d]_{\overline{\pi}_{S*}(f_S)}\ar[r]^(.6){q_S} & R_S & \; & (E_S)\abc{i}\ar[d]_{(f_S)\abc{i}}\ar[r]^(.5){\varphi_{i_S}} & \pi_{\xtilde}^*L_i\\
\overline{\pi}_{S*}(E_S')_{f_S(\{x_1,x_2\}\times S)} \ar[ur]_{q_S'} & & \; & (E_S')\abc{i} \ar[ur]_{\varphi_{i_S}'} & & 
}
\end{equation*}
for $i=1,2$.
\end{deff}

\section{Moduli space}\label{secmodulispace}
Given a descending principal Higgs bundle $\mathfrak E=(E,q,\tautilde,\phitilde)$ on $\xtilde$, if $\tautilde:\symstar(E\otimes V)^G\to\oxtilde$ is zero then $\mathfrak E$ is nothing but a parabolic Higgs vector bundle. These objects are very close to the Higgs vector bundles studied by Simpson in \cite{simpson1994moduli}; on the other hand if $\phitilde:E\to E\otimes \omegaxt$ is zero we get a parabolic principal bundle on a smooth curve and the moduli space of these objects was studied by Schmitt (see \cite{schmitt2005singular}). So the non-trivial case is when $\tautilde$ and $\phitilde$ are both non-zero.\\
As said before, we can consider the family of descending principal Higgs bundles as a subfamily of double decorated generalized parabolic vector bundles with $\det E\simeq\oxtilde$, and for what we saw above, we can assume that $\varphi_1$ and $\varphi_2$ are non-zero morphisms.
If we fix $D\doteqdot\det E$, a morphism $\varphi:E\abc{}\to L$ induces a morphism, that we still call $\varphi$, from $E\ab{}$ to $L\otimes D^{\otimes c}\doteqdot L_{D}$.\\
Let $\tipo\doteqdot(D,r,\underline{a},\underline{b},\underline{c},L_1,L_2)$. The objects we want to classify are quadruples $(E,q,\varphi_1,\varphi_2)$ of type $\tipo$ where $\varphi_i:E\ab{i}\to L_{iD}$ are non-zero morphisms for $i=1,2$.\\
Now we can define functors:\\
\begin{align*}
\moduli_{\tipo}:  & \text{Sch}_{\cc}\to \text{Sets}\\
& \; S\longmapsto\left\lbrace\begin{array}{c}
                  \text{Isomorphism classes of families of}\\
                  (\delta_1,\delta_2)\text{-(semi)stable generalized parabolic}\\
                  \text{vector bundles on $\xtilde$} \\ 
                  \text{of type } \tipo=(D,r,\underline{a},\underline{b},\underline{c},L_1,L_2)\\
                  \text{parametrized by }S
                  \end{array}\right\rbrace
\end{align*}

In order to solve the moduli problem we want to find a scheme $\modulispacess_{\tipo}$ that co-represents the functor $\moduliss_{\tipo}$ and a scheme $\modulispaces_{\tipo}$ that represents the functor $\modulis_{\tipo}$.\\
For this purpose we need to generalize our objects. Indeed, recall that giving a non-zero morphism $\varphi:E_{\rho}\to L$ is the same as giving a section $\sigma :X\to \mathbb{P}(E_{\rho})$. Therefore the morphisms $\varphi_i:E\ab{i}\to L_i\otimes D^{\otimes c_i}$ correspond to morphisms $\sigma_i:X\to \mathbb{P}(E\ab{i})$ for $i=1,2$.\\
Now consider the Segre embedding:
\begin{equation}\label{segre}
\sigma:X\stackrel{(\sigma_1,\sigma_2)}{\longrightarrow}\mathbb{P}(E\ab{1})\times\mathbb{P}(E\ab{2})\to\mathbb{P}(E_{\chi})
\end{equation}
where $\chi=\rho\ab{1}\otimes \rho\ab{2}$ is a homogeneous representation. Observing that $\chi=\rho_{a_1+a_2,b_1b_2}\doteqdot\rho\ab{}$ for $a=a_1+a_2$ and $b=b_1b_2$, one has that $\sigma$ induces a morphism $\varphi:E\ab{}\to L$ for $L=L_{1D}\otimes L_{2D}\simeq L_1\otimes L_2\otimes D^{\otimes (c_1+c_2)}$. So a quadruple $(E,q,\varphi_1,\varphi_2)$ can be viewed as a decorated generalized parabolic vector bundle, with just one decoration. These objects were widely studied by Schmitt in \cite{schmitt2005singular}, where he constructed their moduli spaces with respect to the following definition of (semi)stability:
\begin{deff}[\textbf{(Semi)stable decorated bundles}]\label{def-semistable-decorated}
Fix $\delta\in\qq_{>0}$. A decorated generalized parabolic vector bundle $(E,q,\varphi)$ where $\varphi:E_{a,b}\to L$ is $\delta$-(semi)stable if and only if 
\begin{equation*}
P(\filtrazione)+\delta \mu_{\rho_{a,b}}(\filtrazione;\varphi)(\geq) 0.
\end{equation*}
\end{deff} 

Now thanks to \eqref{segre} a $\ddgpb$ can be viewed as a decorated generalized parabolic bundle and so we have two definitions of (semi)stability and we have to show that they agree.

\begin{teo}
For $(E,q,\varphi_1,\varphi_2)$ and $\varphi$ defined as before the following conditions are equivalent:
\begin{enumerate}
\item For any weighted filtration $(\filtrazione)$ \\
$P(\filtrazione)+\delta_1\muno+\delta_2\mudu \geq 0$
\item For any weighted filtration $(\filtrazione)$\\
$P(\filtrazione)+\delta\mu\ab{}(\filtrazione ;\varphi)\geq 0$,
\end{enumerate}
where $a=a_1+a_2$, $b=b_1b_2$ and $\delta=\delta_1=\delta_2\in \mathbb{Q}_>0$.
\end{teo}
\begin{proof}
We denote $V_{a_1,b_1}$ and $V_{a_2,b_2}$ by $V^1$ and $V^2$, respectively. We choose bases $\{v^1_i\}_{i\in I}$ and $\{v^2_j\}_{j\in J}$ of $V^1$ and $V^2$ such that the action of $\lambda$ is diagonal, i.e. $\lambda(z)_{\cdot}(v^1_i)=z^{\gamma^1_i}v^2_i$ and $\lambda(z)_{\cdot} (v^2_j)=z^{\gamma^2_j}v^2_j$ for any $z\in \cc^*$, $i=1,\dots, \dim(V^1)$ and $j=1,\dots, \dim(V^2)$. Moreover we suppose that $\gamma^1_1 \leq , \dots,\leq \gamma^1_{s_1}$ and $\gamma^2_1\leq\dots,\leq \gamma^2_{s_2}$. By construction the fiber of $E\ab{}$ is $V^1\otimes V^2$ and we have $\lambda(z)_{\cdot}(v^1_i\otimes v^2_j)=z^{\gamma^1_i + \gamma2_j}v^1_i\otimes v^2_j$. Suppose now that $\muno=-\gamma^1_{i_0}$ and $\mudu=-\gamma^2_{j_0}$, using the notations of section \ref{secsemista} this means that $\sigma_1(x)(v^1_i)=0$ for any $i<i_0$ and $\sigma_2(x)(v^2_j)=0$ for any $j<j_0$. Of course $\sigma(x)(v^1_{i_0}\otimes v^2_{j_0})\neq 0$ and so $\mu\ab{}\geq -( \gamma_{i_0}+\gamma'_{j_0})$, if $\mu\ab{}>-(\gamma_{i_0}+\gamma'_{j_0})$ then since $\gamma_i$ and $\gamma'_j$ are ordered there exist $i_1,j_1$ such that $\sigma(x)(v^1_{i_1}\otimes v^2_{j_1})\neq 0$ with either $i_1<i_0$ or $j_1<j_0$, which is impossible.
\end{proof}

\begin{cor}
A \ddgpb is (semi)stable if and only if its associated \dgpb is so.
\end{cor}

In order to study the moduli space of \ddgpb we can restrict to the \dgpb's. Moreover, recalling that $L=L_1\otimes L_2\otimes D^{\otimes(c_1+c_2)}$, it will turn out that the moduli space $\modulispacess_{(D,r,\underline{a},\underline{b},\underline{c},L_1,L_2)}$ is a closed subscheme of the moduli space $\unomodulispacess_{(D,r,a,b,L)}$ and the latter is a closed subscheme of $\unomodulispacess_{(d,r,a,b,L)}$. For this reason, from now on, we will treat only \dgpb's $(E,q,\varphi)$ of type $(d,r,a,b,L)$ where $E$ is a vector bundle of rank $r$ and degree $d$ and $q$ is the parabolic structure, while $\varphi$ will be a function $E\ab{}\to L$. Moreover the notion of isomorphism between \dgpb's and families of \dgpb's are similar to those of a \ddgpb introduced in Section \ref{sec-famiglie}.\\

In \cite{schmitt2005singular} Schmitt constructs a projective scheme $\unomodulispacess_{(d,r,a,b,L)}$  and an open subscheme $\unomodulispaces_{(d,r,a,b,L)}$ which are moduli spaces for the following functors:
\begin{align*}
\mathfrak{\underline{M}}(\rho)^{(\delta)\mbox{-}(s)s}_{d,r,a,b,L}:  & \text{Sch}_{\cc}\to \text{Sets}\\
& \; S\longmapsto\left\lbrace\begin{array}{c}
                  \text{Isomorphism classes of families of}\\
                  \delta\text{-(semi)stable decorated parabolic}\\
                  \text{vector bundles on $\xtilde$} \\ 
                  \text{of type } (d,r,a,b,L)\\
                  \text{parametrized by }S
                  \end{array}\right\rbrace
\end{align*}

\begin{proof}[Proof of Theorem \ref{teo-main}]
There is a one to one correspondence between the family of singular principal Higgs $G$-bundles over $X$ and the family of descending principal Higgs $G$-bundles over $\xtilde$. We have the following chain of inclusions:\\

\begin{displaymath}
\begin{array}{cc}
\text{Objects} & \text{Semistability notion}\\
& \\
\left\lbrace\begin{array}{c}
            \text{Semistable principal Higgs}\\
            G\text{-bundles }(E,\tau,\phi)\text{ on }X\\
            \text{of rank }r
            \end{array}
\right\rbrace & 
\begin{array}{c}
L(E^{\bullet}_{\beta},\underline{\alpha})(\geq)0\\
\text{for any }\phi\text{-invariant }\beta\text{-filtration}\\
\end{array}\\
\downarrow & \\
\left\lbrace\begin{array}{c}
            \text{Semistable singular principal}\\
            \text{Higgs }G\text{-bundles }(\efascio,\tau,\phi)\text{ on }\xtilde\\
            \text{of rank }r
            \end{array}
\right\rbrace & 
\begin{array}{c}
L(\efascio^{\bullet}_{\beta},\underline{\alpha})(\geq)0\\
\text{for any }\phi\text{-invariant }\beta\text{-filtration}\\
\end{array}\\
\updownarrow & \\
\left\lbrace\begin{array}{c}
            \text{Semistable descending principal}\\
            \text{Higgs }G\text{-bundles }(E,q,\tautilde,\phitilde)\text{ on }\xtilde\\
            \text{of rank }r
            \end{array}
\right\rbrace & 
\begin{array}{c}
P(\efil_{\beta},\alphafil_{\beta})(\geq)0\\
\text{for any }\phi\text{-invariant }\beta\text{-filtration}\\
\end{array}\\
\downarrow & \\
\left\lbrace\begin{array}{c}
            \text{Semistable \ddgpb}\\
            (E,q,\varphi_1,\varphi_2)\text{ on }\xtilde\\
            \text{of type }(\oxtilde,r,\underline{a},\underline{b},\underline{c},L_1,L_2)
            \end{array}
\right\rbrace & 
\begin{array}{c}
P(\filtrazione)+\delta_1\muno+\\
+\delta_2\mudu(\geq)0\\
\text{for any weighted filtration }\\
\end{array}\\
\downarrow & \\
\left\lbrace\begin{array}{c}
            \text{Semistable \dgpb}\\
            (E,q,\varphi)\text{ on }\xtilde\\
            \text{of type }(0,r,a,b,L)
            \end{array}
\right\rbrace & 
\begin{array}{c}
P(\filtrazione)+\delta\mu(\filtrazione;\varphi)(\geq)0\\
\text{for any weighted filtration }
\end{array}\\
\end{array}
\end{displaymath}
Therefore the moduli space $\mathring{\singprincipalmodulispacess}$ of isomorphism classes of semistable principal Higgs $G$-bundles over $X$ is a subscheme of $\unomodulispacess_{0,r,a,b,L}$. Then considering its closure $\singprincipalmodulispacess$ in $\unomodulispacess_{0,r,a,b,L}$, thanks to Theorem $2.5$ of \cite{schmitt2000universal}, we get our thesis.
\end{proof}

\subsubsection*{Semistable $n$-uples}

As an application of our results we can construct the moduli space of \textit{semistable $n$-uples} which generalize the semistable pairs introduced by Nitsure in \cite{nitsure1991moduli}.

In this work Nitsure studied the family of semistable pairs $(E,\phi)$ and contructed their moduli space. A pair $(E,\phi)$ consists of a vector bundle $E$ of degree $d$ and rank $r$ on a Riemann surface $X$ and a morphism of vector bundles $\phi:E\to E\otimes L$ where $L$ is a fixed line bundle on $X$. A pair $(E,\phi)$ is said to be (semi)stable if $\mu(F)(\leq) \mu(E)$ for any non-zero $\phi$-invariant subbundle $F\subset E$.

We want to generalize this notion to $n$-uples $(E,\phi_1,\dots,\phi_{n-1})$ where $E$ is a vector bundle on $X$ and $\phi_i:E\to E\otimes L_i$ for all $i=1,\dots,n-1$. We say that a $n$-uple $(E,\phi_1,\dots, \phi_{n-1})$ is \textbf{(semi)stable} if and only if $\mu(F)\leq \mu(E)$ for any non-zero proper $\phi_i$-invariant subbundle $F\subset E$, i.e., a subbundle $F$ such that $0\subsetneq F\subsetneq E$ and $\phi_i(F)\subset F\otimes L_i$ for any $i=1,\dots n-1$.

If some $\phi_{i_0}$ is the zero map we can consider the $n$-uple $(E,\phi_1,\dots,\hat{\phi}_{i_0},\dots,\phi_{n-1})$ obtained from the previous one by deleting the morphism $\phi_{i_0}$. Since every subbundle is $\phi_{i_0}$-invariant the definitions of semistability of the two objects coincide. So we can assume without loss of generality that all $\phi_i$ are non-zero.

By using an argument similar to that used at the beginning of Section \ref{secmodulispace} the $(n-1)$-uple of morphism $(\phi_1,\dots,\phi_{n-1})$ induces a morphism $\phi:E\to E\otimes L$, where $L$ is the tensor product of the $L_i$. So we obtain a pair $(E,\phi)$ instead of a $n$-uple and the (semi)stability conditions are equivalent. Therefore we are reduced to the previous problem.\\

\section{Frame decorated bundles}
Although it is possible to construct the moduli space of decorated bundles, a notion of Jordan-H\"older filtration is still missing. For this reason we will introduce a notion of semistability analogous to the one used for framed sheaves. This semistability, that we will call \textit{frame semistability}, implies the usual semistability condition for decorated bundles given above.\\

Let $(E,q,\varphi)$ be a decorated generalized parabolic bundle over $\xtilde$ with decoration of type $(d,r,a,b,L)$.
We define:
\begin{equation*}
\muphipar{E}{\varphi}\doteqdot \mupar(E)-\delta a \frac{\varepsilon(\varphi)}{\rk{E}},
\end{equation*}
where $\mupar(E)$ is as in Definition \ref{def-semistable-par-higgs-bundles} and
\begin{equation*}
\varepsilon(\varphi)\doteqdot
\begin{cases}
0 \quad\text{if}\;\varphi\equiv0\\
1 \quad\text{if}\;\varphi\not\equiv0.
\end{cases}
\end{equation*}

\begin{deff}[\textbf{\textbf{fr}-(semi)stable decorated bundles}]
A decorated generalized parabolic bundle $(E,q,\varphi)$ over $\xtilde$ with decoration of type $(d,r,a,b,L)$ is (semi)stable if and only if for all subsheaves $F$ of $E$ the following inequality holds:
\begin{equation*}
\muphipar{E}{\varphi}-\muphipar{F}{\rest{\varphi}{F}}(\geq)0.\\
\end{equation*}
\end{deff}

We want to relate this semistability condition with that given by Schmitt.

\begin{prop}\label{prop-additivity}
Let $(E,q,\varphi)$ as before. The following conditions are equivalent:
\begin{enumerate}
 \item $P(\filtrazione)+\delta\mu_{\rho\ab{}}(\filtrazione;\varphi)\geq0$ for all weighted filtrations $(\filtrazione)$,
 \item $\degpar(E)\rk{F}-\degpar(F)\rk{E}+\delta\mu(F,E)\geq0$ for all subsheaves $F\subset E$,
\end{enumerate}
where we recall that $\mu(F,E)=\mu_{\rho\ab{}}(0\subset F\subset E,(1);\varphi)$.
\end{prop}
\begin{proof}
\begin{itemize}
\item[]$(1\Rightarrow 2)$ Obvious.
\item[]$(2\Rightarrow 1)$ Let $(\filtrazione)$ be a weighted filtration. Then condition $(2)$ implies that:
\begin{equation*}
 P(\filtrazione)+\delta\sum_{i=1}^s \alpha_i\mu(E_i,E)\geq0.
\end{equation*}
Since the rapresentation $Gl(V)\to Gl(V_{a,b})$ satisfies the additivity property given by Schmitt in Section 3.1 of \cite{schmitt2000universal}, it follows that $\sum_{i=1}^s \alpha_i\mu(E_i,E)=\mu_{\rho\ab{}}(\filtrazione;\varphi)$ and we are done.
\end{itemize}
\end{proof}

\begin{oss}
 The previous proposition tells us that it is enought to check the semistability condition only on subbundles instead of filtrations. Let us denote with
\begin{equation*}
F^{i_1,\dots,i_k}\doteqdot E\otimes\dots\otimes E\otimes\stackrel{i_1\text{\tiny{-term}}}{F}\otimes E\otimes\dots\otimes E\otimes\stackrel{i_k\text{\tiny{-term}}}{F}\otimes E\otimes\dots\otimes E
\end{equation*}
Finally we define
\begin{equation*}
 \kk{F}{E}\doteqdot\max\{k\in\nn\;|\;\exists (i_1,\dots,i_k)\in\{1,\dots,a\}^k \text{ such that }\rest{\varphi}{(F^{i_1,\dots,i_k})^{\oplus b}}\neq 0\}
\end{equation*}
In this case a calculation shows that
\begin{equation*}
\mu(F,E)=\kk{F}{E}\,\rk{E}-a\,\rk{F},
\end{equation*}
Therefore we can rewrite condition $(2)$ of Proposition \ref{prop-additivity} as follows:
\begin{equation*}
 \mupar(F)-\delta\frac{\kk{F}{E}}{\rk{F}}\leq\mupar(E)-\delta\frac{a}{\rk{E}}.\\
\end{equation*}
\end{oss}
\begin{cor}
 \textbf{fr}-semistability implies Schmitt semistability for decorated bundles.
\end{cor}
\begin{proof}
 Follows directly from Proposition \ref{prop-additivity} and the previous remark.
\end{proof}
From now on with ``(semi)stable'' we will always mean frame (semi)stable.
\subsection{Jordan-Holder filtrations}\label{sec-j-h-fil}
In order to construct Jordan-H\"older filtrations we need the notion of \textbf{quotients for decorated bundles}. If $(F,\rest{\varphi}{F\ab{}})$ is a subsheaf of $(E,\varphi)$ such that $\rest{\varphi}{F}\equiv0$ then we can construct the decorated quotient bundle $(E/F,\overline{\varphi})$, where $\overline{\varphi}:E\ab{}/F\ab{}\longrightarrow L$, induced by $\varphi$, is well defined. Otherwise if $\rest{\varphi}{F\ab{}}\not\equiv0$ we set $\overline{\varphi}=0$ (see \cite{huybrechts1995framed} for details). In general we will say that a decorated parabolic bundle $(E',q',\varphi')$ is a quotient of $(E,q,\varphi)$ if and only if exists $F\subseteq E$ such that $(E',\varphi')\simeq(E/F,\overline{\varphi})$ and $q'$ is induced by $q$.

\begin{lem}\label{lem-mu-baricentrica}
Given an exact sequence of decorated bundles
\begin{equation*}
0\longrightarrow (F,\rest{\varphi}{F})\longrightarrow (E,\varphi) \longrightarrow (E/F,\overline{\varphi}) \longrightarrow 0,
\end{equation*}
one has
\begin{equation*}
\muphipar{E}{\varphi}=\frac{\rk{F}\muphipar{F}{\rest{\varphi}{F}}+\rk{E/F}\muphipar{E/F}{\overline{\varphi}}}{\rk{E}}.
\end{equation*}
\end{lem}
\begin{proof}
It is enough to show that $\degpar(E,\varphi)\doteqdot\degpar(E)-a\delta\varepsilon(\varphi)$ is additive, i.e.
\begin{equation*}
\degpar(E,\varphi)=\degpar(F,\rest{\varphi}{F})+\degpar(E/F,\overline{\varphi}),
\end{equation*}
and this is an easy computation.
\end{proof}

\begin{lem}\label{lem-schur}
Let $(E,q,\varphi)$ and $(E',q',\varphi')$ be two stable decorated parabolic bundles with $\muphipar{E}{\varphi}=\muphipar{E'}{\varphi'}$ and let $f:(E,q,\varphi)\to (E',q',\varphi')$ be a morphism of decorated bundles. Then or $f\equiv0$ either $f$ is an isomorphism.
\end{lem}

\begin{prop}[\textbf{Jordan-H\"older filtration}]
Let $(E,q,\varphi)$ a semistable decorated generalized parabolic bundle, there exist filtrations $\efil:\; 0=E_0\subset E_1\subset\dots\subset E_s\subset E_{s+1}=E$ such that all the factors $E_i/E_{i-1}$ together with the induced parabolic structures and decorations $\overline{\varphi}_i$ are stable and $\muphipar{E}{\varphi}=\muphipar{E_i/E_{i-1}}{\overline{\varphi}_i}$. Any such filtration is called a Jordan-H\"older filtration (or J-H filtration) of $(E,q,\varphi)$.
\end{prop}
\begin{proof}
See \cite{huybrechts1995framed} Proposition 1.13.
\end{proof}
Given a J-H filtration $\efil$ of $(E,q,\varphi)$ we define 
\begin{equation*}
\gr(E)\doteqdot\bigoplus_{i=1}^{s+1} E_i/E_{i-1}.
\end{equation*}
Let us note that
\begin{equation*}
\gr(E)\ab{}=\bigoplus_{i=1}^{s+1} (E_i/E_{i-1})\ab{} \oplus W
\end{equation*}
for a suitable sheaf $W$. Then we define
\begin{equation*}
\gr(\varphi)\doteqdot
\begin{cases}
\bigoplus_{i=1}^{s+1}\overline{\varphi}_i\quad\text{over}\;\bigoplus_{i=1}^{s+1} (E_i/E_{i-1})\ab{}\\
\hspace{0.5cm}0\hspace{1cm}\text{over}\; W.
\end{cases}
\end{equation*}
\begin{equation*}
\gr(q)\doteqdot\bigoplus_{i=1}^{s+1}\overline{q}_i:\gr(E)_{x_1}\oplus\gr(E)_{x_2}\longrightarrow\gr(R)
\end{equation*}
where $\gr(R)\doteqdot\bigoplus_{i=1}^{s+1} R_i/R_{i-1}$ and $R_i$ is a vector space of dimension $\rk{E_i}$ for any $i=0,\dots,s+1$.
\begin{prop}
Let $(E,q,\varphi)$ be a semistable decorated parabolic bundle with decoration $(d,r,a,b,L)$. The decorated parabolic bundle $(\gr(E),\gr(q),\gr(\varphi))$ with decoration $(d,r,a,b,L)$ does not depend on the J-H filtration chosen.
\end{prop}
\begin{proof}
By induction on $\rk{E}$. If $\rk{E}=1$ there is nothing to prove, otherwise we assume the statement is true for any $r'<r=\rk{E}$ and we prove it for $r$.
Let $\efil:\; 0\subset E_1\subset\dots\subset E_s\subset E_{s+1}=E$ and $F^{\bullet}:\; 0\subset F_1\subset\dots\subset F_s\subset F_{t+1}=E$ be two J-H filtrations of $E$. Let $j$ be the smallest index such that $E_1\subset F_j$. Then
\begin{equation*}
\psi:E_1\longrightarrow F_j\longrightarrow F_j/F_{j-1}
\end{equation*}
is a nontrivial homomorphism of stable vector bundles with the same $\textbf{fr}$-slope and so, by Lemma \ref{lem-schur}, $\psi$ is an isomorphism. Moreover there is a short exact sequence of decorated bundles 
\begin{equation*}
0\to (F_{j-1},q,\varphi)\to (E/E_1,q,\varphi)\to (E/F_j,q,\varphi)\to 0,
\end{equation*}
abusing $q,\varphi$ as a generic notation for the induced morphisms. The J-H filtrations of $E/F_j$ and $F_{j-1}$ give rise to a J-H filtration of $E/E_1$, whose graded object, by induction on the rank of $E$, is isomorphic to the graded object of the filtration $\efil/E_1$. 
\end{proof}
\begin{deff}
Two semistable decorated generalized parabolic bundles $(E,q,\varphi)$ and $(E',q'\varphi')$ with $\muphipar{E}{\varphi}=\muphipar{E'}{\varphi'}$ are called $S$-equivalent if and only if $(\gr(E),\gr(q),\gr(\varphi))\simeq(\gr(E'),\gr(q'),\gr(\varphi'))$
\end{deff}


\subsection{Construction of the moduli space}
In this section we construct the moduli space of semistable decorated bundles with fixed rank and degree and we prove that it is a projective scheme that contains the moduli space of stable decorated bundles as an open subset. If the decoration is trivial then we get the moduli space of semistable vector bundles on $\xtilde$. Therefore we assume that the decoration is non-zero.
Since the family $\ssfam$ of isomorphism classes of semistable generalized parabolic bundles with decoration $(d,r,a,b,L)$ is bounded (\cite{schmitt2005singular} Section 3), we can find a integer $m_0$ such that $\forall m\geq m_0$ and for all class $E\in\ssfam$ we have that:
\begin{enumerate}
\item $H^1(E(m))=0$,
\item $H^0(E(m))$ is globally generated,
\end{enumerate}
where we recall that $E(m)=E\otimes\oxtilde(m)$. Let $E\in\ssfam$ and let $(E,q,\varphi)$ be a representant for the class of $E$, then for a generic subsheaf $E'\subseteq E$ the Hilbert polynomial is:
\begin{equation*}
P_{E'}(l)\doteqdot \rk{E'}l+\deg(E')+(1-g)\rk{E'},
\end{equation*}
moreover we define
\begin{equation*}
\poly{E'}\doteqdot P_{E'}-\dim q(E'_{x_1}\oplus E'_{x_2})-a\;\delta\;\varepsilon(\rest{\varphi}{E'}).
\end{equation*}
Let $\quot$ be the projective scheme that parametrizes quotients $\quoz:Y\otimes\oxtilde(-m)\to E$, where $Y$ is a vector space of dimension $h^0(E(m))$ and $E$ is a coherent sheaf of degree $d$ and rank $r$. For sufficiently large $l$ the standard morphism
\begin{equation*}
\xymatrix{
 \quot\ar[r] & \text{Grass}(Y\otimes H^0(\oxtilde(l-m)),P_{E}(l))\ar[d]\\
& \pp(\bigwedge^{rl+d+(1-g)r}Y\otimes H^0(\oxtilde(l-m)))
}
\end{equation*}
is a well-defined closed immersion (note that $P_{E}(l)=rl+d+(1-g)r=h^0(E(l))$).
Let $Y_1=(Y\otimes\oxtilde(-m))_{x_1}$ and $Y_2=(Y\otimes\oxtilde(-m))_{x_2}$ be the stalks at the points $x_1$, $x_2$ and denote by $\gras\doteqdot\mbox{Gr}(Y_1\oplus Y_2,r)$ the Grassmannian that parametrizes $r$-dimensional quotients $Y_1\oplus Y_2\to R$. Giving a morphism $f:Y\ab{}\otimes\oxtilde(-m)\to L$ is equivalent to giving a morphism, that we will still call $f$, from $Y\ab{}$ to $H^0(L(m))$. Now, let $\textbf{P}\doteqdot\pp(\home(Y\ab{},H^0(L(m)))^{\vee})$ and $Z'\subset\quot\times\gras\times\textbf{P}$ denote the closed subscheme of points
\begin{equation*}
 ([\quoz:Y\otimes\oxtilde(-m)\to E],[g:Y_1\oplus Y_2\to R],[f:Y\ab{}\to H^0(L(m))])
\end{equation*}
for which the induced homomorphisms $f:Y\ab{}\otimes\oxtilde(-m)\to L$ and the morphism $g:Y_1\oplus Y_2\to R$  make the following diagrams commutative:
\begin{equation*}
\xymatrix{ Y\otimes\oxtilde(-m) \ar[d]\ar[r]^(0.65){\quoz} & E \ar[d] & \; & Y\otimes\oxtilde(-m) \ar[d]\ar[r]^(0.65){\quoz} & E \ar[d]\\
Y\ab{}\otimes\oxtilde(-m) \ar[d]_{f} \ar[r]^(0.65){\quoz\ab{}} & E\ab{} \ar[dl]^{\varphi}& \; & Y_1\otimes Y_2 \ar[r]^\quoz\ar[d]_g & E_{x_1}\oplus E_{x_2} \ar[dl]^q\\
L & & \; & R &\\
}
\end{equation*}
The immersion $\quot\hookrightarrow\pp(\bigwedge^{h^0(E(l))}Y\otimes H^0(\oxtilde(l-m)))$ and the immersion $\gras\hookrightarrow\pp(\bigwedge^r (Y_1\oplus Y_2))$ give rise to very ample line bundles $\mathcal{L}_{\quot}$ over $\quot$ and $\mathcal{L}_{\gras}$ over $\gras$. Then the group $Sl(Y)$ acts on $Z'$ and the line bundles
\begin{equation*}
\mathcal{O}_{Z'}(n_1,n_2,n_3)\doteqdot p^*_{\quot}\mathcal{L}_{\quot}^{\otimes n_1}\otimes p^*_{\gras}\mathcal{L}_{\gras}^{n_2}\otimes p^*_{\textbf{P}}\mathcal{O}_{\textbf{P}}(n_3)
\end{equation*}
where $p_{\quot}$, $p_{\gras}$ and $p_{\textbf{P}}$ are the projections from $\quot\times\gras\times\textbf{P}$ to $\quot$, $\gras$ and $\textbf{P}$ respectively, carry natural $Sl(Y)$-linearizations. In the following we choose $n_1,n_2,n_3$ such that:
\begin{equation*}
\frac{n_2}{n_1}=\frac{\polyf(l)}{\polyf(m)}-1,\qquad \frac{n_3}{n_1}=\frac{\delta\,\polyf(l)}{\polyf(m)}-\delta,
\end{equation*}
where we recall that $\polyf(l)=P_E(l)-r-a\;\delta$.
\begin{oss}\label{H-L-1-rmkss}
With this notation the semistability condition for $(E,q,\varphi)$ is equivalent to requiring that $\frac{\poly{E'}}{\rk{E'}}\leq \frac{\polyf}{r}$ for  any non-zero proper subsheaf $E'\subset E$ or $\frac{\polygen{E''}{q''}{\varphi''}}{\rk{E''}}\geq \frac{\polyf}{r}$ for any quotient sheaf $E''$.
\end{oss}
\begin{prop}
For sufficiently large $l$ the point $([\quoz],[g],[f])\in Z'$ is (semi)stable with respect to the linearization of $\mathcal{O}_{Z'}(n_1,n_2,n_3)$ if and only if the following condition holds: if $Y'$ is a nontrivial proper subspace of $Y$ and $E'\subset E$ the subsheaf generated by $Y'\otimes\oxtilde(-m)$, then
\begin{equation}\label{prop-H-L-1-eq2}
\dim Y' \; (n_1 P_E(l)+rn_2+a n_3)\; (\leq) \dim Y \; (n_1 P_{E'}(l)+n_2\qdim{E'}+a n_3\varepsilon(\rest{\varphi}{E'})),
\end{equation}
where $\qdim{E'}\doteqdot\dim q(E'_{x_1}\oplus E'_{x_2})$ and $\varphi$ is the morphism induced by $f$.
\end{prop}
\begin{proof}
Let $\quoz:Y\otimes\oxtilde(-m)\to E$, $g:Y_1\oplus Y_2\to R$ and $f:Y\ab{}\to H^0(L(m))$ be homomorphisms representing the point $([\quoz],[g],[f])$. We put $P_E(l)=p_l$ for convenience's sake. Let $W=H^0(\oxtilde(l-m))$, $\quoz$ induces homomorphisms $\quoz':Y\otimes W\to H^0(E(l))$ and $\quoz'':\bigwedge^{p_l} (Y\otimes W)\to\bigwedge^{p_l}H^0(E(l))$. If $\{w_1,\dots,w_t\}$ is a basis for $W$ and $\{u_1,\dots,u_k\}$ is a basis for $Y$, then a basis for $\bigwedge^{p_l}(V\otimes W)$ is given by elements of the form
\begin{equation*}
 u_{IJ}=(y_{i_1}\otimes w_{j_1})\wedge\dots\wedge(y_{i_{p_l}}\otimes w_{j_{p_l}})
\end{equation*}
where $I,J$ are multi-indices satisfying $i_{h}\leq i_{h+1}$ and $j_{h}<j_{h+1}$ if $i_{h}=i_{h+1}$. Given a one-parameter subgroup $\lambda$ of $Sl(Y)$ with vector weights $\underline{\xi}=(\xi_1,\dots,\xi_k)$, then $\cc^*$ acts on $\bigwedge^{p_l}(Y\otimes W)$ by
\begin{equation*}
\lambda(t)\ldotp u_{IJ}= t^{\xi_I} u_{IJ}, \qquad \xi_I\doteqdot\sum_{i_h\in I}\xi_{i_h}.
\end{equation*}
Now let $\mu(\lambda;\quoz'')=-\min\{\xi_I\;|\;\exists I,J\;\text{with}\;\quoz''(u_{IJ})\neq0\}$ (note that $\mu(\lambda;\quoz'')=\mu(\lambda;\quoz)$ where $\mu(\lambda;\quoz)$ is defined in Section \ref{secsemista}). This number can be computed as follows. Let $\varpi$ denote the function $t\longmapsto\dim \quoz'(<u_1,\dots,u_t>\otimes W)$, then
\begin{equation*}
\mu(\lambda;\quoz'')=-\sum_{i=1}^{k}\xi_i\;(\varpi(i)-\varpi(i-1)).
\end{equation*}
We recall that a one-parameter subgroup gives a weighted filtration $(Y^{\bullet},\underline{\alpha})$ where $Y^{\bullet}:\quad 0\subset Y_1\subset\dots\subset Y_{s}\subset Y_{s+1}=Y$. We set
\begin{equation*}
\mu\ab{}^{\fr}(\lambda;f)=-\min\{a\,\gamma_i\;|\; \rest{f}{(Y_{i})\ab{}\otimes\oxtilde(-m)}\neq0\}
\end{equation*}
with
\begin{equation*}
\underline{\gamma}\equiv\underline{\gamma}(\lambda)=\sum_{i=1}^s\alpha_i(\underbrace{\dim Y_i-k,\dots,\dim Y_i-k}_{(\dim Y_i)\text{-times}},\underbrace{\dim Y_i,\dots,\dim Y_i}_{(k-\dim Y_i)\text{-times}})
\end{equation*}
(see Section \ref{secsemista} for more details).\\
Let $\tau$ be the index that realizes the minimum of, i.e., $\mu\ab{}^{\fr}(\lambda;f)=-a\gamma_{\tau}$. Finally, we define $\mu(\lambda;g)$ as in Section \ref{secsemista} and recall that a point $g:Y_1\oplus Y_2\to R$ is (semi)stable if and only if for any $Y'\subset Y$
\begin{equation*}
\frac{2\dim Y'}{\dim g(Y'_1\oplus Y'_2)}=\frac{\dim Y'_1\oplus Y'_2}{\dim g(Y'_1\oplus Y'_2)}(\leq)\frac{\dim Y_1\oplus Y_2}{\dim g(Y_1\oplus Y_2)}=\frac{2\dim Y}{r},
\end{equation*}
where we put $Y'_i\doteqdot (Y'\otimes\oxtilde(-m))_{x_i}$ for $i=1,2$.\\

Now, recalling that $\dim Y=k$, from the Hilbert-Mumford criterion (\cite{mumford1994geometric} Theorem 2.1) we have:\\

\textit{$([\quoz],[g],[f])$ is a (semi)stable point of $Z'$ if and only if for all one-parameter subgroups $\lambda$ one has
\begin{equation*}
n_1\;\mu(\lambda;\quoz'')+n_2\;\mu(\lambda;g)+n_3\;\mu\ab{}^{\fr}(\lambda;f)\;(\geq)\;0,
\end{equation*}
or equivalently,
\begin{equation}\label{prop-H-L-1-eq}
n_1\sum_{i=1}^{k}\xi_i\;(\varpi(i)-\varpi(i-1))+n_2\left( r\dim Y'-k\dim g(Y'_1\oplus Y'_2)\right)+n_3\;\gamma_{\tau}\;(\leq)\;0.
\end{equation}
}

The left hand side is a linear form of weight vectors whose coefficients are determined only by the choice of the basis. Keeping such a basis fixed for a moment, it is enough to check the inequality for the special one-parameter subgroups $\lambda^{(i)}$ giving the weight vectors
\begin{equation*}
\underline{\xi}^{(i)}=(\underbrace{i-k,\dots,i-k}_{i\text{-times}},\underbrace{i,\dots,i}_{(k-i)\text{-times}}) \qquad i=1,\dots,k-1.
\end{equation*}
Note that in this case the filtration induced by $\lambda^{(i)}$ has length one and weight vector $\underline{\alpha}=(\underline{1})$, so $\underline{\gamma}(\lambda^{(i)})=\underline{\xi}^{(i)}$.
For $\underline{\xi}^{(i)}$ the inequality \eqref{prop-H-L-1-eq} is equivalent to
\begin{equation*}
i\; (n_1p_l+n_2 r+a n_3)(\leq)k(n_1\varpi(i)+n_2\dim g(Y'_1\oplus Y'_2)+a n_3\varepsilon(i)),
\end{equation*}
where
\begin{equation*}
\varepsilon(i)=
\begin{cases}
1 \qquad \text{if } \rest{f}{((<y_1,\dots,y_i>)^{\otimes a})^{\oplus b}\otimes\oxtilde(-m)}\neq 0\\
0 \qquad \text{otherwise}.
\end{cases}
\end{equation*}
Then the following holds:\\

\textit{$([\quoz],[g],[f])$ is a (semi)stable point of $Z'$ if and only if for all non trivial proper subspaces $Y'$ of $Y$ one has
\begin{align*}
\dim Y' & \; (n_1p_l+n_2 r+an_3)\quad(\leq)\nonumber\\
& (\leq)\, \dim Y \;(n_1\dim\quoz'(Y'\otimes W)+n_2\dim g(Y'_1\oplus Y'_2)+an_3\varepsilon(Y')),
\end{align*}
where
\begin{equation*}
\varepsilon(Y')=
\begin{cases}
1 \qquad \text{if } \rest{f}{Y'\ab{}}\neq 0\\
0 \qquad \text{otherwise}.
\end{cases}
\end{equation*}
}

Let $E'$ the subbundle $\quoz(Y'\otimes\oxtilde(-m))$. In this case the decoration $\varphi:E\ab{}\to L$ vanishes when restricted to $E'$ if and only if $\rest{f}{Y'\ab{}\otimes\oxtilde(-m)}=0$. Hence $\varepsilon(Y')=\varepsilon(\rest{\varphi}{E'})$ and recalling that $P_E(l)=p_l$ and $P_{E'}(l)=\dim\quoz'(Y'\otimes W)$ we are done.
\end{proof}

Now consider a \dgpb $(E,q,\varphi)$. If $E$ is torsion-free then so are $E\ab{}$ and $\ker\varphi\subset E\ab{}$. Conversely, if $\ker\varphi$ is torsion-free, from the following exact sequence:
\begin{equation*}
0\longrightarrow\ker\varphi\longrightarrow E\ab{}\stackrel{\varphi}{\longrightarrow} L\longrightarrow 0
\end{equation*}
we obtain that, since $L$ is locally-free, $E\ab{}$ cannot have torsion either and so $E$ is torsion-free. Therefore we have that
\begin{align*}
\text{The family of \dgpb with tors} & \text{ion free kernel}=\nonumber\\
& =\text{The family of torsion free \dgpb's}
\end{align*}
and then by \cite{maruyama1976openness} Proposition 2.1 we have the following result:
\begin{lem}
If $(E_S,q_S,\varphi_S)$ is a flat family of \dgpb's parametrized by a Noetherian scheme $S$, then the subset of points $s\in S$ such that $E_s$ is torsion free is open in $S$.
\end{lem}

Let $U\subset Z'$ the open subscheme consisting of those points that represent torsion free \dgpb, and let $Z\doteqdot \overline{U}$ be the closure in $Z'$ of $U$.

\begin{prop}
For sufficiently large $l$, a point $([\quoz],[g],[f])\in Z$ is (semi)stable with respect to the $Sl(Y)$ action on $Z$ is and only if the corresponding decorated parabolic bundle $(E,q,\varphi)$ is (semi)stable and $\quoz$ induces an isomorphism $Y\to H^0(E(m))$.
\end{prop}
\begin{proof}
Observe that if $([\quoz],[g],[f])$ is a semistable point the homomorphism $Y\to H^0(E(m))$ must be injective. Indeed if $Y'$ is its kernel than $\quoz'(Y'\otimes W)=0$, $\dim g(Y'_1\oplus Y'_2)=0$ and $\varepsilon(Y')=0$, so the previous proposition shows that $\dim(Y')=0$. Hence, since $\dim Y=h^0(E(m))$, the morphism $Y\to H^0(E(m))$ is an isomorphism.\\
Substituting 
\begin{equation*}
\frac{n_2}{n_1}=\frac{\polyf(l)}{\polyf(m)}-1\qquad\frac{n_3}{n_1}=\frac{\delta\,\polyf(l)}{\polyf(m)}-\delta
\end{equation*}
in the inequality \eqref{prop-H-L-1-eq2} and setting $Y'\doteqdot Y\cap H^0(E'(m))$ for any non-trivial proper subbundle $E'$ of $E$ we can rewrite the stability criterion \eqref{prop-H-L-1-eq2} as follows:\\

\textit{$([\quoz],[g],[f])$ is a (semi)stable point of $Z$ if and only if for all non trivial proper subsheaf $E'$ of $E$ with the induced decoration the following holds:
\begin{align}\label{prop-H-L-2-eq}
& \dim Y'\,\polyf(l)\; \left( 1+\frac{a\delta+r}{\polyf(m)}\right)\quad(\leq)\nonumber\\
& (\leq) \dim Y  \left( \poly{E'}(l)+\left[a\delta\varepsilon(\rest{\varphi}{E'})+\qdim{E'}\right]\frac{\polyf(l)}{\polyf(m)} \right)
\end{align}
}

Recalling that by definition $P_E'(m)=h^0(E'(m))$, $P_E(m)=\polyf(m)+a\delta+r$ and that
\begin{equation*}
\dim Y'=\poly{E'}(m)+a\delta\varepsilon(\rest{\varphi}{E'})+\qdim{E'},
\end{equation*}
the previous inequality is equivalent to the following:
\begin{align*}
&\left(\poly{E'}(m)+ a\delta\varepsilon(\rest{\varphi}{E'})+\qdim{E'}\right)\polyf(l)\left(\frac{P_E(m)}{\polyf(m)}\right)(\leq)\nonumber\\
& P_E(m)\left(\frac{\polyf(m)\poly{E'}(l)+\left[a\delta\varepsilon(\rest{\varphi}{E'})+\qdim{E'}\right]\polyf(l)}{\polyf(m)}\right)
\end{align*}
and after some simplifications we get
\begin{equation*}
\polyf(l)\;\poly{E'}(m)(\leq)\polyf(m)\;\poly{E'}(l).
\end{equation*}
Since the inequality \eqref{prop-H-L-2-eq} holds for every $l$ large enough the same inequality holds also for the coefficients of $l$. Then we derive the inequality:
\begin{equation}\label{prop-H-L-2-eq2}
r\poly{E'}(m)(\leq)\rk{E'}\polyf(m),
\end{equation}
and then $E$ is (semi)stable.\\

Conversely if $(E,q,\varphi)$ is (semi)stable then for any non trivial proper subsheaf $E'$ of rank $r'$ one has $\poly{E'}(m)\leq r'\;\polyf(m)/r$. The previous inequality is equivalent to \eqref{prop-H-L-2-eq}, and we are done.
\end{proof}

\begin{teo}
There exists a projective scheme $\unomodulispacessfr$ and a morphism $\pi:Z^{ss}\to\unomodulispacessfr$ which is a good quotient for the action of $Sl(Y)$ on $Z^{ss}$. Moreover there is an open subscheme $\unomodulispacesfr\subset\unomodulispacessfr$ such that $Z^s=\pi^{-1}(\unomodulispacesfr)$ and $\pi:Z^s\to \unomodulispacesfr$ is a geometric quotient. Two points $([\quoz_1],[g_1],[f_1])$ and $([\quoz_2],[g_2],[f_2])$ are mapped to the same point in $\unomodulispacessfr$ if and only if the corresponding decorated bundles are $S$-equivalent. 
\end{teo}
\begin{proof}
The proof, with the necessary modifications, is quite similar to that one of Theorem 3.3 in \cite{huybrechts1995framed}.
\end{proof}

\begin{teo}\label{teo-sp-moduli}
There exists a projective scheme $\unomodulispacessfr_{d,r,a,b,L}$ with is a coarse moduli space for the functor $\unomodulissfr_{d,r,a,b,L}$ which associates to a scheme $T$ the set of isomorphism classes of flat families of semistable decorated generalized parabolic vector bundles defined over $T$ of type $(d,r,a,b,L)$. Moreover, there is a open subscheme $\unomodulispacesfr_{d,r,a,b,L}$ which is a fine moduli space for the subfunctor $\unomodulisfr_{d,r,a,b,L}$ of families of stable \dgpb's. The closed points of $\unomodulispacessfr_{d,r,a,b,L}$ represent semistable \dgpb's up $S$-equivalence.
\end{teo}
\begin{proof}
Let $(E_T,q_T,\varphi_T)$ be a flat family of \dgpb parametrized by a scheme $T$ and $m$ the number fixed before. Then $\mathcal{Y}\doteqdot p_{T*}(E_T\otimes P_X^*\oxtilde(m))$ is a locally free sheaf of rank $P_E(m)$ on $T$ and $p_T^*\mathcal{Y}\to E_T$ is surjective by Proposition 8.5 and 8.8 Chapter 3 of \cite{hartshorne1977algebraic}. Moreover the parabolic structure $q_T$ induces a morphism
\begin{equation*}
g_T:\overline{\pi}_{T*}\left((p_T^*\mathcal{Y})_{\{x_1\}\times T}\oplus (p_T^*\mathcal{Y})_{\{x_2\}\times T}\right)\longrightarrow R_T
\end{equation*}
and the decoration $\varphi_T$ induces a homomorphism
\begin{equation*}
f_T:\mathcal{Y}\ab{}\longrightarrow\mathcal{O}_T\otimes H^0(L(m)).
\end{equation*}
From now on the proof is the same as in Theorem 0.1 \cite{huybrechts1995framed}.
\end{proof}

\thanks{\textit{Acknowledgement}. The authors would like to thank Professor Ugo Bruzzo, for several discussions and explanations, and Professor Usha N. Bhosle for her interesting suggestions.}

                                    %
\bibliography{Biblio-des.bib}
\bibliographystyle{plain}
\end{document}